\documentclass[12pt]{article} 
\usepackage[colorlinks]{hyperref}
\usepackage{amsmath}
\usepackage{amsthm}
\usepackage{amsfonts}
\usepackage{amssymb}
\usepackage[mathscr]{eucal}
\usepackage{bussproofs}
\usepackage{lineno}
\usepackage{quiver}

\usepackage{quiver}
\usepackage{tikzsymbols}

\usepackage{tikz}

\usepackage{authblk}

\theoremstyle{plain} 
\newtheorem{theorem}{Theorem}[section]

\newtheorem{lemma}[theorem]{Lemma}

\newtheorem{corollary}[theorem]{Corollary}
\theoremstyle{definition}
\newtheorem{definition}[theorem]{Definition}
\newtheorem{example}[theorem]{Example}
\newtheorem{remark}[theorem]{Remark}

\begin{document}

\title{On a Generalization of Heyting Algebras I}

\author[1]{Amirhossein Akbar Tabatabai}
\author[2]{Majid Alizadeh}
\author[2]{Masoud Memarzadeh}
\affil[1]{Institute of Mathematics, Czech Academy of Sciences}
\affil[2]{School of Mathematics, Statistics and Computer Science, College of Science, University of Tehran}

\date{ }

\maketitle

\begin{abstract}
$\nabla$-algebra is a natural generalization of Heyting algebra, unifying many algebraic structures including bounded lattices, Heyting algebras, temporal Heyting algebras and the algebraic presentation of the dynamic topological systems.
In a series of two papers, we will systematically study the algebro-topological properties of different varieties of $\nabla$-algebras. In the present paper, we start with investigating the structure of these varieties by characterizing their subdirectly irreducible and simple elements. Then, we prove the closure of these varieties under the Dedekind-MacNeille completion and provide the canonical construction and the Kripke representation for $\nabla$-algebras by which we establish the amalgamation property for some varieties of $\nabla$-algebras. In the sequel of the present paper, we will complete the study by covering the logics of these varieties and their corresponding Priestley-Esakia and spectral duality theories.\\


\end{abstract}

\section{Introduction}

\subsubsection*{Abstract Implications}

Implication is the logical machinery to \emph{internalize} the \emph{meta-linguistic} provability relation defined \emph{over} a language to a logical constant \emph{inside} the language itself. Generally speaking, the internalization is a process to reflect \emph{some} properties of the relation $\vdash$ by the corresponding properties of its internal version
$\to$.
For instance, the reflexivity of $\vdash$, i.e., the property $A \vdash A$ is reflected by the property $\vdash A \to A$ and the transitivity of $\vdash$, i.e., the property ``\emph{$A \vdash B$ and $B \vdash C$ implies $A \vdash C$}" is reflected by the property $(A \to B) \wedge (B \to C) \vdash (A \to C)$.
Depending on the structures available over the propositions and the extent to which the internalization process reflects the corresponding properties, various types of implications appeared in various logical systems. The spectrum spreads from the standard instances of classical, intuitionistic and substructural implications to  philosophically motivated conditionals \cite{NuteCross}, weak sub-intuitionistic \cite{Vi2,visser1981propositional,Ru} and mathematically motivated implications appearing in provability logic \cite{visser1981propositional} and preservability logic \cite{Iem1,Iem2,LitViss}. To unify all these instances under one umbrella notion, some formalizations have been proposed \cite{Celani-Jansana,Lattar}. Here, we follow the abstract notion of implication introduced in \cite{ImSpace} as a binary operation performing the minimum amount of internalization, i.e., internalizing the fact that the provability relation is actually a preorder. More precisely, an implication $\to$ is a binary operation over a meet semi-lattice $\mathsf{A}$, decreasing in its first argument and increasing in its second satisfying the \emph{internal reflexivity} $(a \to a)=1$ and the \emph{internal transitivity} $(a \to b) \wedge (b \to c) \leq (a \to c)$, for any $a, b, c \in \mathsf{A}$.\footnote{Technically, the definition presented in \cite{ImSpace} uses a monoidal poset rather than a meet semi-lattice to cover the substructural implications. However, in this paper, we restrict ourselves to this special case as we are only interested in the so-called structural framework.}

Among all possible implications, the Heyting implication and its defining adjunction have a special mathematical and philosophical position. The reason is the role that the adjunction plays in the internalization process by \emph{reducing} the internal version of a property into the property itself. Let us explain this role by a concrete example.
Consider the existence of the binary meets identified by the \emph{equivalence} between ``$a \leq b \wedge c$" and ``(\emph{$a \leq b$ and $a \leq c$})" and note that its internal version is the \emph{equality} $(a \to (b \wedge c))=(a \to b) \wedge (a \to c)$.
To prove the internal version for the Heyting implication, we show the equivalence between $d \leq a \to (b \wedge c)$ and $d \leq (a \to b) \wedge (a \to c)$, for an arbitrary $d$. By the adjunction, $d \leq a \to (b \wedge c)$ is equivalent to $d \wedge a \leq b \wedge c$. By the original property of the existence of the meets, this is equivalent to ($d \wedge a \leq b$ and $d \wedge a \leq c$) and by the adjunction again and using the properties of the meet, we reach $d \leq (a \to b) \wedge (a \to c)$. It is clear that the adjunction reduces the internal version of the existence of the binary meets to its original version. 
The examples of such reductions abound. For instance, think about the property $(a \vee b) \to c=(a \to c) \wedge (b \to c)$ for the Heyting implication as the internalization of the existence of the binary joins or $(a \to (b \supset c))=(a \wedge b) \to c$, as the internalization of the existence of the Heyting implication $\supset$. 

Compared to the well-behaved Heyting implication, an abstract implication is known to be rather ill-behaved. The reason, we believe, is the lack of the adjunction and the natural reduction process it provides. Without a reduction procedure, we are forced to select the internalized properties ourselves and the choice can be quite artificial if we do not have a general guiding principle in the selection process. For instance, internalizing the structure of a bounded distributive lattice, one may wonder if the internalization of both the existence of the binary meets and the binary joins, i.e., satisfying $(a \to (b \wedge c))=(a \to b) \wedge (a \to c)$ and $(a \vee b) \to c=(a \to c) \wedge (b \to c)$ is what one must expect or having one without the other is also meaningful. The reader probably shares our feeling that the answer to this problem is quite far from clear. \\
Apart from the importance of the reduction processes, there are also other and probably more conceptual arguments in favor of the adjunctions. The most important one is the role that the adjunction plays in specifying the meaning of a logical constant. To explain, an adjuncton is nothing but a full pair of introduction and elimination rules providing a full characterization of how the constant can be introduced on the one hand and how it can be used on the other. This full characterization specifies the behavior of the logical constant which we can consider as the (proof-theoretic) meaning of the constant. In this sense, the lack of the adjunction leads to an incomplete determination of the meaning of a connective and it is no surprise that such a connective is ill-behaved to some extent. 

To overcome this lack of adjunction, \cite{ImSpace} introduced $\nabla$-algebras\footnote{In \cite{ImSpace}, $\nabla$-algebras are called temporal algebras and they are again based on monoidal posets to also cover the substructural instances.} as a generalization of Heyting algebras importing many of their nice properties into the realm of the abstract implications. 
Formally, a $\nabla$-algebra is a tuple $(\mathsf{A}, \nabla, \to)$, where $\mathsf{A}$ is a bounded lattice over the set $A$, $\nabla: A \to A$ is a unary operation and $\to \, : A \times A \to A$ is a binary operation such that $\nabla(-) \wedge a \dashv a \to (-)$, i.e., $\nabla(-) \wedge a$ is left adjoint to $a \to (-)$, for any $a \in A$.  $\nabla$-algebras are the clear generalization of both Heyting algebras ($\nabla=id$) and bounded lattices ($\nabla a=0$ and $a \to b=1$). It is easy to see that in any $\nabla$-algebra, $\to$ is an abstract implication empowered by the adjunction to run a natural reduction process on the one hand and provide a complete proof-theoretic meaning for the implication, on the other.
The main feature of $\nabla$-algebras, however, is their ability to fuse the full generality of the abstract implications to the nice behavior of the Heyting implications. To be more precise, $\nabla$-algebras provide not only a nice family of implications but also a sufficiently general one powerful enough to represent all abstract implications 
up to a correcting factor,
see \cite{ImSpace} and also Theorem \ref{NablaImpliesStrong} in the present paper. Having this representation theorem, it is then mathematically justified to employ the strategy of studying the well-behaved $\nabla$-algebras in order to understand all abstract implications. This is exactly the task that the present series of two papers is devoted to. Before explaining how, let us complete the list of motivations for $\nabla$-algebras by providing two other minor yet interesting instances. 
\subsubsection*{Intuitionistic Temporal Logics}
Temporal logics form an interesting family of modal logics addressing the temporal aspects of the natural language. The intuitionistic version of these logics addressing the combinations of knowledge and various temporal modalities including ``past" and ``present" \cite{Ewald,de2017constructive}, ``next" \cite{Davies2017,Koj}, ``henceforth" \cite{Kamide} and their combinations \cite{Maier,balbiani2018bisimulations,boudou2017decidable} is also introduced and investigated in the literature.
The most basic system among these logics is the one only consisting of one \emph{existential past} modality $\Diamond$ and one \emph{universal future} modality $\Box$. Putting algebraically, the logic is nothing but a \emph{temporal Heyting algebra}, i.e., a Heyting algebra augmented with an adjunction $\Diamond \dashv \Box$ \cite{ModEsakia,Chajda}. Many extensions of temporal Heyting algebras, their Sahlqvist correspondence theory and their canonicity are studied, e.g. see \cite{Ghil,Pal}. \\
Having this temporal setting in mind, one can read $\nabla$-algebras as a generalization of temporal Heyting algebras weakening the base from a Heyting algebra to a bounded lattice and strengthening the unary $\Box$ to a binary $\to$. Here, $\nabla a$ is interpreted similar to $\Diamond a$ as ``\emph{$a$ was true at some point in the past}"  and $a \to b$ means ``\emph{$a$ implies $b$, always in the future}". As expected, when the base lattice of a $\nabla$-algebra has the Heyting structure, we regain our temporal Heyting algebras. 
The reason is the fact that in the presence of the Heyting implication $\supset$, the implication of a $\nabla$-algebra satisfies $a \to b=1 \to (a \supset b)$. Therefore, setting $\Box a=1 \to a$, the structure of a Heyting $\nabla$-algebra simplifies to a Heyting algebra augmented with the adjunction $\nabla \dashv \Box$. Having this observation, it is safe to say that $\nabla$-algebras genuinely generalize temporal Heyting algebras and hence their study encompasses the algebraic study of the basic intuitionistic temporal logic. However, it is also important to keep in mind that $\nabla$-algebras are more general than temporal Heyting algebras as in lack of the Heyting implication, the binary implication $\to$ is more expressive and forces more structure than the unary $\Box$. To see why, compare the structure of a $\nabla$-algebra to a bounded lattice augmented with an adjunction $\nabla \dashv \Box$. Trivializing $\nabla$ to the identity function turns them into a Heyting algebra and a bounded lattice, respectively, where the former has far more structure than the latter. 



\subsubsection*{The Logic of Dynamic Topological Systems}
A dynamic topological system is a pair $(X, f)$, where $X$ is a topological space formalizing the \emph{state space} of the system and $f: X \to X$ is a continuous map encoding its \emph{dynamism} \cite{Akin}. To describe the dynamic topological systems, \cite{artemov1997modal} enhanced the logical system $\mathsf{S4}$ by a ``next" modality $\bigcirc$,
where the former described the topological space and the latter captured the effect of the inverse image of the map $f$. Later,
\cite{DTL} enriched the language further to make it expressive enough to address the asymptotic behavior of the systems, as well. The logic turned out to be undecidable \cite{Und}. In the quest of finding a rich enough decidable fragment and following an unpublished suggestion of Kremer \cite{kremer2004small},  Fern\'{a}ndez-Duque \cite{Fer} introduced the intuitionistic version of the dynamic topological logic, where $\mathsf{S4}$ is replaced by the intuitionistic propositional logic to describe the topological state space.
Inspired by \cite{Fer}, it is natural to claim that a very basic algebraic setting to formalize the dynamic topological system $(X, f)$ is a Heyting algebra representing the locale $\mathcal{O}(X)$ augmented with a modality $\nabla$ representing the map $f^{-1}$. To identify the properties of this modality, by imitating the behavior of $f^{-1}$, we expect $\nabla$ to preserve the finite meets and the arbitrary joins. The former is easy to axiomatize algebraically. To capture the latter in an elementary way, first note that as $f^{-1}$ preserves all the joins over the complete lattice $\mathcal{O}(X)$, it must have a right adjoint $f_*$. Thus, to reflect the behavior of $f^{-1}$ in an elementary way, it is enough to also add $\Box$ to the signature to mimic the map $f_*$. Therefore, the structure we are looking for is simply a Heyting algebra augmented with an adjunction $\nabla \dashv \Box$. As we saw before, such a structure is nothing but a Heyting $\nabla$-algebra, this time with the additional condition that $\nabla$ must respect all the finite meets. Calling this property the \emph{normality} condition, we can conclude that normal Heyting $\nabla$-algebras provide an algebraic presentation for the dynamic topological systems. Using this observation, one can expand the intuition to read any normal $\nabla$-algebra (not necessarily Heyting) as some sort of dynamic system where the base lattice captures the observables of the system and $\nabla$ encodes its dynamism. 

\subsubsection*{Our Contribution}

Inspired by the above motivations, in a series of two papers, we provide an extensive algebro-topological study of different varieties of 
$\nabla$-algebras. In the present paper, after covering the preliminary concepts in Section \ref{Prel} and introducing some varieties of $\nabla$-algebras in Section \ref{IntroducingAlgebras}, we study the structure of these varieties in Section \ref{SimpleAlg} and provide a characterization for subdirectly irreducible and simple normal distributive $\nabla$-algebras. In Section \ref{Dedekind}, we cover the Dedekind-MacNeille completion of $\nabla$-algebras to show how well-behaved the notion of a $\nabla$-algebra is. We will also need this completion to prove stronger forms of completeness in the sequel of the present paper. Section \ref{Kripke} is devoted to Kripke frames and the usual canonical construction providing an interesting representation theorem for distributive normal $\nabla$-algebras as the \emph{dynamic posets}, i.e, the posets equipped with an order-preserving map over them. The representation then helps to prove the amalgamation property for some varieties of normal distributive $\nabla$-algebras. This completes the scope of the present paper. In its sequel, we introduce some logical systems for which $\nabla$-algebras provide an algebraic semantics. We prove the deductive interpolation property for some of these logics. We also provide Kripke and topological semantics for them and prove a corresponding soundness and completeness theorem. We will continue by providing a full duality theory for the distributive $\nabla$-algebras, unifying Priestley and Esakia dualities as the two instances of one duality theory. This duality theory then leads to a corresponding spectral duality \`{a} la \cite{Bezh}. We will use the latter to provide a ring-theoretic representation for normal distributive $\nabla$-algebras interpreting the algebra as a \emph{dynamic ring}, i.e., a ring equipped with a ring endomorphism. 

\section{Preliminaries} \label{Prel}

In this section, we will present the preliminary concepts and their corresponding theorems that we need throughout the paper. In the algebraic part that is our main emphasis, we will be quite extensive. In the topological part, though, we only recall some useful points.

A \emph{poset} $\mathcal{P}=(P, \leq)$ is a pair of a set and a reflexive, anti-symmetric and transitive binary relation $\leq \; \subseteq P^2$. By $\mathcal{P}^{op}$, we mean the poset $(P, \geq)$, namely $P$ with the opposite order. A subset $S \subseteq P$ is called a \emph{downset}, if it is downward closed, i.e., for every $x \in S$ and any $y \in P$, if $y \leq x$, we have $y \in S$. An \emph{upset} is defined dually. For any subset $S \subseteq P$, by its downset, denoted by $\downarrow \!\!S$, we mean the least downset extending $S$ and by the upset of $S$, denoted by $\uparrow \!\!S$, we mean the least upset extending $S$. When $S=\{a\}$, we denote $\downarrow \!\!S$ and $\uparrow \!\!S$ by $\downarrow \!\!a$ (or $(a]$) and $\uparrow \!\!a$ (or $[a)$), respectively. The set of all upsets and downsets of $(P, \leq)$ are denoted by $U(P, \leq)$ and $D(P, \leq)$, respectively. For any $S \subseteq P$, if the greatest lower bound of $S$ exists, it is called the \emph{meet} of the elements of $S$ and is denoted by $\bigwedge S$. If $S=\{a, b\}$, the meet $\bigwedge S$ is usually denoted by $a \wedge b$. A poset is called \emph{complete} if for any set $S \subseteq P$, the meet $\bigwedge S$ exists. If the least upper bound of $S$ exists, it is called the \emph{join} of the elements of $S$ and is denoted by $\bigvee S$. If $S=\{a, b\}$, the join $\bigvee S$ is usually denoted by $a \vee b$. In any complete poset all the joins also exist. A poset is called a \emph{meet semi-lattice}, if for any $a, b \in P$, the meet $a \wedge b$ exists. It is called bounded if it also has a greatest element, denoted by $1$. A meet semi-lattice is called a \emph{lattice}, if for any $a, b \in P$, the join $a \vee b$ also exits. A lattice is called \emph{bounded}, if it has both the greatest and the least elements, denoted by $1$ and $0$, respectively. A bounded lattice is called \emph{distributive}, if $
a \wedge (b \vee c)=(a \wedge b) \vee (a \wedge c)$,
for any $a, b, c \in P$. A bounded lattice $(P, \leq)$ is called a \emph{locale}, if it is complete and $a \wedge \bigvee_{b \in S} b =\bigvee_{b \in S} (a \wedge b)$, for any $a \in P$ and any $S \subseteq P$.
We will use the sanserif font as in $\mathsf{A}$ for the bounded lattices and for brevity, when we are working with a bounded lattice $\mathsf{A}$, without any further explanation, we use the normal $A$ as the base set of $\mathsf{A}$.
A subset of a bounded lattice is called a \emph{filter}, if it is an upset and closed under all finite meets. The filters in the form $[a)$ are called the \emph{principal filters}. The set of all filters of the lattice $\mathsf{A}$ is denoted by $\mathcal{F}(\mathsf{A})$. A filter is called \emph{prime} if $0 \notin P$ and $a \vee b \in P$ implies either $a \in P$ or $b \in P$. The set of all prime filters of a lattice $\mathsf{A}$ is denoted by $\mathcal{F}_p(\mathsf{A})$. A subset of $\mathsf{A}$ is called an \emph{ideal}, if it is a downset and closed under all finite joins. The ideals in the form $(a]$ are called the \emph{principal ideals}. The set of all ideals of a lattice $\mathsf{A}$ is denoted by $\mathcal{I}(\mathsf{A})$. The following theorem is a useful tool when the bounded lattice is distributive.
\begin{theorem}\label{PrimeFilterTheorem}\cite{Esakia}(Prime filter theorem)
Let $\mathsf{A}$ be a distributive lattice, $F$ be a filter and $I$ be an ideal such that $F \cap I=\varnothing$. Then there exists a prime filter $P$ such that $F \subseteq P$ and $P \cap I=\varnothing$.
\end{theorem}
Let $(P, \leq_P)$ and $(Q, \leq_Q)$ be two posets and $f: P \to Q$ be a function. It is called a \emph{order-preserving map}, if it preserves the order, meaning $f(p) \leq_Q f(q)$ for any $p \leq_P q$. An order-preserving map is called an \emph{order embedding}, if for any $p, q \in P$, the inequality $f(p) \leq_Q f(q)$ implies $p \leq_P q$. An order-preserving map between two meet semi-lattices is called a \emph{meet semi-lattice map}, if it preserves the binary meets. Similarly, an order-preserving map is called a \emph{bounded meet-semilattice map}, if it also preserves the greatest element. It is called a lattice map, if it preserves both meets and joins and it is a \emph{locale map} if it preserves all finite meets and all possible joins. 
For two order-preserving maps $f: P \to Q$ and $g: Q \to P$, the map $f$ is called the \emph{left adjoint} for $g$ (and $g$ is called the \emph{right adjoint} for $f$) denoted by \emph{$f \dashv g$} iff
\[
f(a) \leq_Q b \;\;\; \text{iff} \;\;\; a \leq_P g(b),
\]
for any $a \in P$ and $b \in Q$. In such a situation, the pair $(f, g)$ is called an \emph{adjunction}. We will use the following two easy theorems throughout the paper.

\begin{theorem}\cite{Bor1}\label{InjSurjforAdj}
Let $(P, \leq_P)$ and $(Q, \leq_Q)$ be two posets and $f: P \to Q$ and $g: Q \to P$ be two order-preserving maps such that $f \dashv g$. Then:
\begin{itemize}
\item
$fg(q) \leq_Q q$ and $p \leq_P gf(p)$, for any $p \in P$ and $q \in Q$.
\item
$fgf(p)=f(p)$ and $gfg(q)=g(q)$, for any $p \in P$ and $q \in Q$. 
\item
$f$ is one-to-one  iff $g$ is surjective iff $gf(p)=p$, for any $p \in P$.
\item
$f$ is surjective iff $g$ is one-to-one iff $fg(q)=q$, for any $q \in Q$.
\end{itemize}
\end{theorem}

\begin{theorem} \cite{Bor1} (Adjoint functor theorem for posets) \label{AdjointFunctorTheorem}
Let $(P, \leq_P)$ be a complete poset and $(Q, \leq_Q)$ be a poset. Then an order-preserving map $f: (P, \leq_P) \to (Q, \leq_Q)$ has a right (left) adjoint iff it preserves all joins (meets). 
\end{theorem}
Let $X$ be a topological space. We denote the locale of its open subsets by $\mathcal{O}(X)$. A topological space is called $T_0$, if for any two different points $x, y \in X$, there is an open set which contains one of these points and not the other. It is called $T_D$, if for any $x \in X$, there is an open $U$ such that $x \in U$ and $U-\{x\}$ is open, as well. A continuous map is called a topological embedding if $f$ induces a homeomorphism between $X$ and $f[X]$. The following provides a point-free version of surjectivity and embeddingness.

\begin{theorem}\label{InjSurjforContinuous} \cite{StoneSpaces}
Let $X$ and $Y$ be two topological spaces and $f: X \to Y$ be a continuous map. Then, 
\begin{itemize}
\item
If $f$ is surjective, then $f^{-1}: \mathcal{O}(Y) \to \mathcal{O}(X)$ is one-to-one. The converse is true, if $Y$ is $T_D$.
\item
If $f$ is a topological embedding, then $f^{-1}: \mathcal{O}(Y) \to \mathcal{O}(X)$ is surjective. The converse is also true, if $X$ is $T_0$.
\end{itemize}
\end{theorem}

\section{$\nabla$-Algebras}\label{IntroducingAlgebras}
In this section, we first introduce  $\nabla$-algebras as the main notion of the present paper. Then, we present some interesting families of $\nabla$-algebras and two canonical types of examples to connect them to basic temporal Heyting algebras and dynamic topological systems. We then proceed by proving some basic properties of the introduced families and showing that they all form a variety.
Finally, we recall the notion of an abstract implication and its representation by $\nabla$-algebras. As already mentioned in Introduction, this representation serves as the main motivation for $\nabla$-algebras. 

\subsection{The Basics of $\nabla$-algebras} \label{TheBasicsofNablaAlgebra}
Let us start this subsection with the definition of a $\nabla$-algebra:

\begin{definition}
Let $\mathsf{A}=(A, \wedge, \vee, 0, 1)$ be a bounded lattice with the lattice order $\leq$. A tuple $\mathcal{A}=(\mathsf{A}, \nabla, \to)$ is called a $\nabla$-algebra if $\nabla c \wedge a \leq b$ is equivalent to $c \leq a \to b$, for any $a, b, c \in A$. In any $\nabla$-algebra, $\Box a$ is defined as $1 \to a$.
\end{definition}

\begin{example}
Any bounded lattice $\mathsf{A}$ with the operations $\nabla$ and $\to$ defined by setting $\nabla a=0$ and $a \to b=1$, for any $a, b \in A$, forms a $\nabla$-algebra, because
\[
0=\nabla c \wedge a \leq b \;\;\; \text{iff} \;\;\; c \leq a \to b=1.
\] 
To have another example, any Heyting algebra with $\nabla a=a$ and $a \to b=a \supset b$, where $\supset$ is the Heyting implication, forms a $\nabla$-algebra. Therefore, $\nabla$-algebras are a common generalization of the bounded lattices and the Heyting algebras. 
\end{example}
Before moving to the more interesting examples, let us first introduce some specific families of $\nabla$-algebras. Later, in Example \ref{TopologicalSpace} and Example \ref{KripkeFrames}, we will see the concrete situations from which these families are abstracted.
\begin{definition}
Let $\mathcal{A}=(\mathsf{A}, \nabla, \to)$ be a $\nabla$-algebra:
\begin{description}
\item[$(D)$]
If $\mathsf{A}$ is distributive, then the $\nabla$-algebra is called \emph{distributive}.
\item[$(H)$]
If $\mathsf{A}$ has the Heyting structure, then the $\nabla$-algebra is called \emph{Heyting}. This does not mean that $\to$ is the Heyting implication. It only dictates that the Heyting implication exists over $\mathsf{A}$.
\item[$(N)$]
If $\nabla$ commutes with all finite meets, i.e., $\nabla 1=1$ and $\nabla (a \wedge b)=\nabla a \wedge \nabla b$, for any $a, b \in A$, then the $\nabla$-algebra is called \emph{normal}.
\item[$(R)$]
If $a \leq \nabla a$, for any $a \in A$, the $\nabla$-algebra is called \emph{right}.
\item[$(L)$]
If $\nabla a \leq a$, for any $a \in A$, the $\nabla$-algebra is called \emph{left}.
\item[$(Fa)$]
If $\nabla$ is surjective, the $\nabla$-algebra is called \emph{faithful}.
\item[$(Fu)$]
If $\Box$ is surjective, the $\nabla$-algebra is called \emph{full}.
\end{description}
\end{definition}

\begin{remark}
The Heyting implication is a structure and not a mere property. Therefore, ``Heyting $\nabla$-algebra" is an ambiguous notion, as it is not clear whether we include the Heyting implication in the signature of the algebra or not. To solve this ambiguity, when we mean a $\nabla$-algebra that is also Heyting, we call it ``\emph{Heyting $\nabla$-algebra}" and when we mean an algebra in the form $(\mathsf{A}, \nabla, \to, \supset)$, where $(\mathsf{A}, \nabla, \to)$ is a $\nabla$-algebra and $\supset$ is the Heyting implication over $\mathsf{A}$, we call the structure ``\emph{explicitly Heyting $\nabla$-algebra}". This difference in the signature is important when we investigate the algebraic or categorical properties of the algebras.
\end{remark}

\begin{definition}
For any $C \subseteq \{D, H, N, R, L, Fa, Fu\}$, by $\mathcal{V}(C)$ we mean the class of all $\nabla$-algebras with the properties described in the set $C$. For instance, $\mathcal{V}(\{N, D\})$ is the class of all normal distributive $\nabla$-algebras. Sometimes, for simplicity, we omit the brackets. For instance, we write $\mathcal{V}(N, D)$ for $\mathcal{V}(\{N, D\})$ and if $X \in \{D, H, N, R, L, Fa, Fu\}$, we write $\mathcal{V}(X, C)$ for $\mathcal{V}(\{X\} \cup C)$. By $\mathcal{V}_H(C)$, we mean the class of all explicitly Heyting $\nabla$-algebras satisfying the conditions in $C$. For two $\nabla$-algebras $\mathcal{A}=(\mathsf{A}, \nabla_{\mathcal{A}}, \to_{\mathcal{A}})$ and $\mathcal{B}=(\mathsf{B}, \nabla_{\mathcal{B}}, \to_{\mathcal{B}})$, by a \emph{$\nabla$-algebra morphism}, we mean a bounded lattice morphism $f: \mathsf{A} \to \mathsf{B}$ that also preserves $\nabla$ and $\to$, meaning that $f(\nabla_{\mathcal{A}} a)=\nabla_{\mathcal{B}} f(a)$ and $f(a \to_{\mathcal{A}} b)=f(a) \to_{\mathcal{B}} f(b)$, for any $a, b \in A$. If both $\nabla$-algebras are explicitly Heyting and $f$ also preserves the Heyting implication, it is called a \emph{Heyting $\nabla$-algebra morphism}. A $\nabla$-algebra morphism is called an \emph{embedding} if it is injective. For any $C \subseteq \{D, H, N, R, L, Fa, Fu\}$, the class $\mathcal{V}(C)$ of $\nabla$-algebras together with the $\nabla$-algebra morphisms forms a category, denoted by $\mathbf{Alg}_{\nabla}(C)$. Similarly, the class  $\mathcal{V}_H(C)$ of explicitly Heyting $\nabla$-algebras with Heyting morphisms forms a category, denoted by $\mathbf{Alg}^H_{\nabla}(C)$.
\end{definition}

The following two examples are borrowed from \cite{ImSpace} to provide some natural families of $\nabla$-algebras.

\begin{example}\label{TopologicalSpace}\emph{(Dynamic Topological Systems)}
Let $X$ be a topological space and $f: X \to X$ be a continuous function. Then, the pair $(X, f)$ is called a \emph{dynamic topological system}. If $f$ is a topological embedding (resp. surjection), the dynamic topological system is called \emph{faithful} (resp. \emph{full}).  For any dynamic topological system $(X, f)$, define $\to_f$ over $\mathcal{O}(X)$ by $U \to_f V=f_* (int(U^c \cup V))$, where $int$ is the interior operation, $U^c$ is the complement of $U$ and $f_*: \mathcal{O}(X) \to \mathcal{O}(X)$ is the right adjoint of $f^{-1}$. Note that using Theorem \ref{AdjointFunctorTheorem}, the map $f_*$ exists, as $f^{-1}$ preserves the arbitrary unions and $\mathcal{O}(X)$ is complete. Then, the structure $(\mathcal{O}(X), f^{-1}, \to_f)$ is a normal Heyting $\nabla$-algebra, simply because for any $U, V, W \in \mathcal{O}(X)$, we have
\[
f^{-1}(W) \cap U \subseteq V \;\;\; \text{iff} \;\;\; W \subseteq f_*(int(U^c \cup V))=U \to_f V,
\] 
and $f^{-1}$ commutes with finite intersections. It is also worth mentioning that adding the Heyting implication $\supset$ to the structure results in the explicitly Heyting normal $\nabla$-algebra $(\mathcal{O}(X), f^{-1}, \to_f, \supset)$. Moreover, using Theorem \ref{InjSurjforAdj} and Theorem \ref{InjSurjforContinuous}, we can observe that if $(X, f)$ is full, i.e., $f$ is surjective, then the algebra $(\mathcal{O}(X), f^{-1}, \to_f)$ is full, as the surjectivity of $f$ implies the injectivity of $\nabla=f^{-1}$ and hence the surjectivity of $f_*$. The converse also holds if $X$ is $T_D$. In addition, if $(X, f)$ is faithful, i.e., $f$ is a topological embedding, then the $\nabla$-algebra $(\mathcal{O}(X), f^{-1}, \to_f)$ is faithful, as being an embedding implies the surjectivity of $\nabla=f^{-1}$. The converse also holds if $X$ is $T_0$.  
\end{example}

As shown in Example \ref{TopologicalSpace}, any dynamic topological system gives rise to a normal $\nabla$-algebra over a locale. The converse is not necessarily true, as the locale may not be the locale of the opens of a topological space. However, this is the only obstacle and moving to the point-free setting, these $\nabla$-algebras actually coincide with the \emph{point-free dynamic topological systems}, i.e., the pairs $(\mathscr{X}, F)$ where $\mathscr{X}$ is a locale and $F: \mathscr{X} \to \mathscr{X}$ is a locale map. The reason simply is the fact that in any normal $\nabla$-algebra $(\mathscr{X}, \nabla, \to)$ over the locale $\mathscr{X}$, the $\nabla$ is both join-preserving and finite meet-preserving and hence it is a locale map over $\mathscr{X}$ and conversely, as soon as $\nabla$ is a locale map over $\mathscr{X}$, its adjoint, i.e., $\to$ exists by the adjoint functor theorem, Theorem \ref{AdjointFunctorTheorem} and it is unique. Therefore, the normal $\nabla$-algebra $(\mathscr{X}, \nabla, \to)$ over the locale $\mathscr{X}$ is nothing but the point-free dynamic topological system $(\mathscr{X}, \nabla)$.
Having that observed, then $(Fa)$ is the point-free way to state that the dynamism $\nabla$ is a topological embedding, while $(Fu)$ is the point-free version of the surjectivity of $\nabla$. Notice that $(N)$, $(Fa)$ and $(Fu)$ together axiomatize the point-free situation where the dynamism is actually a homeomorphism.

\begin{remark}
Faithfulness is an interesting condition to consider even if the $\nabla$-algebra does not correspond to any dynamic topological system or its base lattice is not even a Heyting algebra. The reason is the pure implicational role that the faithfulness plays. To be more precise, a $\nabla$-algebra is faithful iff $c \to a \leq c \to b$ implies $c \wedge a \leq b$, for any $a, b, c \in A$, see Theorem \ref{Surjectivity}. This property is known to logicians for its philosophical and technical roles, see for instance \cite{Ard2,BasicPropLogic}.
\end{remark}

\begin{example}\label{KripkeFrames}\emph{(Kripke Frames)}
Let $(W, \leq)$ be a poset. By a \emph{Kripke frame}, we mean a tuple $\mathcal{K}=(W, \leq , R)$, where $R$ is a binary relation over $W$ compatible with the partial order, i.e., $\leq \circ \, R \, \circ \leq \, \subseteq R$ or more explicitly, if $k' \leq k$, $l \leq l'$ and $(k, l) \in R$ then $(k', l') \in R$, for any $k, k', l, l'\in W$:
\[\begin{tikzcd}
	k && l \\
	\\
	{k'} && {l'}
	\arrow["\leq", from=3-1, to=1-1]
	\arrow["R", from=1-1, to=1-3]
	\arrow["\leq", from=1-3, to=3-3]
	\arrow["R"', dashed, from=3-1, to=3-3]
\end{tikzcd}\]
To any Kripke frame, we can assign a canonical Heyting $\nabla$-algebra, encoding its structure algebraically. Set $\mathscr{X}$ as the locale of all upsets of $(W, \leq)$ and define $\nabla: \mathscr{X} \to \mathscr{X}$ as $\nabla_{\mathcal{K}} U=\{x \in W \mid \exists y \in U, (y, x) \in R \}$ and $U \to_{\mathcal{K}} V=\{x \in W \mid \forall y \in W [(x, y) \in R \wedge y \in U \Rightarrow y \in V]\}$. It is easy to see that $(\mathscr{X}, \nabla_{\mathcal{K}}, \to_{\mathcal{K}})$ is a Heyting $\nabla$-algebra and hence $(\mathscr{X}, \nabla_{\mathcal{K}}, \to_{\mathcal{K}}, \supset)$ is an explicitly Heyting $\nabla$-algebra, where $\supset$ is the usual Heyting implication, i.e., $U \supset V=\{x \in W \mid \forall y \in W [(x \leq y) \wedge y \in U \Rightarrow y \in V]\}$. It is also easy to see that the $\nabla$-algebra $(\mathscr{X}, \nabla_{\mathcal{K}}, \to_{\mathcal{K}})$ is right (resp. left) if $R$ is reflexive (resp. $R \subseteq \, \leq $).
\end{example}

One can read a Kripke frame as a very basic intuitionistic temporal setting, where the set $W$ is the state space and the relations $\leq$ and $R$ are capturing the growth of the knowledge and time over $W$, respectively. Notice that by this interpretation, the reflexivity condition on $R$ simply says that the order of time is reflexive while the condition $R \subseteq \, \leq$ means that the growth of time implies the growth of knowledge. Having such a temporal setting, the $\nabla$-algebra constructed in Example \ref{KripkeFrames} encodes the structure that the intuitionistic temporal setting imposes on the propositions. Apart from being a Heyting algebra, the structure is simply the pair of the \emph{existential past} modality $\nabla$ and the \emph{universal future} modality $\to$.

\begin{remark}
Left and right conditions are interesting even if the $\nabla$-algebra does not correspond to any Kripke frame or its base lattice is not even a Heyting algebra. The reason again is the pure implicational role that these properties play. To be more precise, the left $\nabla$-algebras are the ones where the introduction rule for the implication is admissible, i.e., $c \wedge a \leq b$ implies $c \leq a \to b$, see Theorem \ref{Left}, and the right $\nabla$-algebras are the ones where the modus ponens is admissible, i.e., $(a \to b) \wedge a \leq b$, see Theorem \ref{Right}. Such implications are investigated in the literature for some philosophical \cite{Ard2,BasicPropLogic,Okada} and technical reasons \cite{visser1981propositional}. 
\end{remark}

In the rest of this subsection, we will investigate some of the basic properties of $\nabla$-algebras. We will then use these properties to show that for any $C \subseteq \{D, N, R, L, Fa, Fu\}$, the classes $\mathcal{V}(C)$ and $\mathcal{V}_H(C)$ form a variety.

\begin{theorem}\label{BasicPropOfNablaAlg} Let $(\mathsf{A}, \nabla, \to)$ be a $\nabla$-algebra. Then:
\begin{description}
    \item[$(i)$]
The operations $\nabla:(A, \leq) \to (A, \leq)$ and $\to: (A, \leq)^{op} \times (A, \leq) \to (A, \leq)$ are order-preserving. Therefore, it is possible to define a $\nabla$-algebra as a structure $(\mathsf{A}, \nabla, \to)$, where $\mathsf{A}$ is a bounded lattice and for any $a \in A$, the operations $\nabla$ and $a \to (-)$ are two order-preserving maps satisfying the adjunction $\nabla (-) \wedge a \dashv a \to (-)$.
    \item[$(ii)$]
$\nabla \dashv \Box$. Therefore, in any $\nabla$-algebra, $\nabla$ (resp. $\Box$) preserves all the existing joins (resp. meets), especially, $\nabla 0=0$, $\nabla(a \vee b)=\nabla a \vee \nabla b$, $\Box 1=1$ and $\Box (a \wedge b)=\Box a \wedge \Box b$, for any $a, b \in A$.
    \item[$(iii)$]
$\nabla (a \to b) \wedge a \leq b$, $\nabla \Box a \leq a$ and $a \leq \Box \nabla a$, for any $a, b \in A$.     
\end{description}
\end{theorem}
\begin{proof}
For $(i)$, if $c \leq d$, for any $b \in A$, we have
\[
\nabla d \leq b \;\;\; \text{implies} \;\;\; d \leq 1 \to b \;\;\; \text{implies} \;\;\; c \leq 1 \to b \;\;\; \text{implies} \;\;\; \nabla c \leq b.
\]
Setting $b=\nabla d$, we reach $\nabla c \leq \nabla d$. A similar argument works for $\to$. The rest are trivial.
\end{proof}

\begin{theorem} \label{Right} (Right) Let $\mathcal{A}=(\mathsf{A}, \nabla, \to)$ be a $\nabla$-algebra. Then, the following are equivalent:
\begin{description}
\item[$(i)$]
$b \leq \nabla b$, for any $b \in A$.
\item[$(ii)$]
$a \wedge (a \to b) \leq b$, for any $a, b \in A$.
\item[$(iii)$]
$\Box b \leq b$,  for any $b \in A$.
\end{description}
\end{theorem}
\begin{proof}
To prove $(ii)$ from $(i)$, apply Theorem \ref{BasicPropOfNablaAlg} to have
$a \wedge (a \to b) \leq a \wedge \nabla (a \to b) \leq b$. To prove $(iii)$ from $(ii)$, put $a=1$ in $(ii)$. To prove $(i)$ from $(iii)$, by Theorem \ref{BasicPropOfNablaAlg}, we have $b \leq \Box \nabla b \leq \nabla b$.
\end{proof}

\begin{theorem} \label{Left} (Left) Let $\mathcal{A}=(\mathsf{A}, \nabla, \to)$ be a $\nabla$-algebra. Then, the following are equivalent:
\begin{description}
\item[$(i)$]
$\nabla b \leq b$, for any $b \in A$.
\item[$(ii)$]
$c \wedge a \leq b$ implies $c \leq a \to b$, for any $a, b, c \in A$.
\item[$(iii)$]
$b \leq \Box b$, for any $b \in A$.
\end{description}
\end{theorem}
\begin{proof}
To prove $(ii)$ from $(i)$, if $c \wedge a \leq b$, then $\nabla c \wedge a \leq b$ and hence $c \leq a \to b$, by adjunction. To prove $(iii)$ from $(ii)$, put $a=1$ and $c=b$ in $(ii)$. To prove $(i)$ from $(iii)$, use the adjunction $\nabla \dashv \Box$ in Theorem \ref{BasicPropOfNablaAlg}.
\end{proof}

\begin{theorem}\label{Surjectivity} (Faithfulness)
Let $\mathcal{A}=(\mathsf{A}, \nabla, \to)$ be a $\nabla$-algebra. Then, the following are equivalent:
\begin{description}
\item[$(i)$]
$\nabla \Box a=a$, for any $a \in A$.
\item[$(ii)$]
$\nabla$ is surjective.
\item[$(iii)$]
$a \wedge \nabla(a \to b)=a \wedge b$, for any $a, b \in A$.
\item[$(iv)$]
$c \to a \leq c \to b$ implies $c \wedge a \leq b$, for any $a, b, c \in A$.
\item[$(v)$]
$\Box$ is an order embedding, i.e., if $\Box a \leq \Box b$ then $a \leq b$, for any $a, b \in A$.
\end{description}
\end{theorem}
\begin{proof}
It is clear that $(i)$ implies $(ii)$. To prove $(iii)$ from $(ii)$, by Theorem \ref{BasicPropOfNablaAlg}, we have $a \wedge \nabla (a \to b) \leq b$. Therefore, it is enough to show that $b \leq \nabla(a \to b)$. Since $\nabla$ is surjective, there exists $c \in A$ such that $b=\nabla c$. Since $b \wedge a \leq b$, we have $\nabla c \wedge a \leq b$ which implies $c \leq a \rightarrow b$ and hence, $b=\nabla c \leq \nabla (a \rightarrow b)$, by the monotonicity of $\nabla$ proved in Theorem \ref{BasicPropOfNablaAlg}. To prove $(iv)$ from $(iii)$, assume $c \to a \leq c \to b$. Hence, $c \wedge \nabla (c \to a) \leq c \wedge \nabla (c \to b)$, again by the the monotonicity of $\nabla$. Therefore, using $(iii)$, we have $c \wedge a \leq b$. To reach $(v)$ from $(iv)$, it is just enough to set $c=1$ in $(iv)$. And finally, to prove $(i)$ from $(v)$, use Theorem \ref{InjSurjforAdj}. 
\end{proof}

\begin{remark} \label{OnOne}
Note that $\nabla 1=1$, in any faithful $\nabla$-algebra. Because, $1= \nabla \Box1=\nabla (1 \to 1)=\nabla 1$. As a consequence, $a \to b=1$ iff $a \leq b$, because $1 \leq a \to b$ iff $\nabla 1 \wedge a \leq b$ iff $a \leq b$.
\end{remark}

\begin{corollary}\label{FaithfulImpliesHeyting}
Any faithful $\nabla$-algebra has a Heyting structure. Consequently, any faithful $\nabla$-algebra is distributive.
\end{corollary}
\begin{proof}
Define $a \supset b=\nabla (a \to b)$, for any $a, b \in A$. We claim that $\supset$ is the Heyting implication. To prove, if $c \wedge a \leq b$, then by part $(i)$ of Theorem \ref{Surjectivity}, we have $\nabla \Box c=c$. Hence, $\nabla \Box c \wedge a \leq b$ which implies $\Box c \leq a \to b$ and then $c=\nabla \Box c \leq \nabla (a \to b)$, by the monotonicity of $\nabla$, Theorem \ref{BasicPropOfNablaAlg}. Conversely, if $c \leq \nabla (a \to b)$, then $c \wedge a \leq \nabla (a \to b) \wedge a \leq b$, by Theorem \ref{BasicPropOfNablaAlg}.
\end{proof}

\begin{theorem}\label{Injectivity} (Fullness)
Let $\mathcal{A}=(\mathsf{A}, \nabla, \to)$ be a $\nabla$-algebra. Then, the following are equivalent:
\begin{description}
\item[$(i)$]
$\Box \nabla a=a$, for any $a \in A$.
\item[$(ii)$]
$\Box$ is surjective.
\item[$(iii)$]
$\nabla$ is an embedding, i.e., if $\nabla a \leq \nabla b$ then $a \leq b$, for any $a, b \in A$.
\end{description}
\end{theorem}
\begin{proof}
See Theorem \ref{InjSurjforAdj}.
\end{proof}

\begin{theorem}\label{Variety}
The following set of equalities and inequalities axiomatize the class of all $\nabla$-algebras:
\begin{description}
\item[$(1)$]
The set of equalities axiomatizing the variety of bounded lattices,
\item[$(2)$]
$(a \wedge b) \rightarrow a=1$.
\item[$(3)$]
$\nabla (a \wedge b) \leq \nabla a \wedge \nabla b$.
\item[$(4)$]
$a \wedge \nabla (a \rightarrow b) \leq b$.
\item[$(5)$]
$c \wedge [(\nabla c \wedge a) \to b] \leq a \rightarrow b$.
\end{description}
Therefore, the class $\mathcal{V}$ is a variety. More generally, for any subset $C \subseteq \{D, N, R, L, Fa, Fu\}$, the classes $\mathcal{V}(C)$ and $\mathcal{V}_H(C)$ are varieties.
\end{theorem}
\begin{proof}
To prove the soundness of the axiomatization, we have to show that all the axioms in $(1)-(5)$ are satisfied in any $\nabla$-algebra. The axioms in $(1)$ hold as any $\nabla$-algebra is a bounded lattice and $(2)-(4)$ are trivially satisfied using the adjunction and the monotonicity of $\nabla$, see Theorem \ref{BasicPropOfNablaAlg}. For $(5)$, by the adjunction, it is enough to show that $\nabla [c \wedge [(\nabla c \wedge a) \to b]] \wedge a \leq b$. By the monotonicity of $\nabla$, we have 
$
\nabla [c \wedge [(\nabla c \wedge a) \to b]] \leq \nabla c \wedge \nabla [(\nabla c \wedge a) \to b]
$
and since 
\[
\nabla c \wedge \nabla [(\nabla c \wedge a) \to b] \wedge a= \nabla [(\nabla c \wedge a) \to b] \wedge (\nabla c \wedge a) \leq b,
\]
we are done.
For the completeness, assume that $(\mathsf{A}, \nabla, \to)$ satisfies $(1)-(5)$. First, notice that by the axioms in $(1)$, $\mathsf{A}$ is a bounded lattice. Secondly, note that if $x \leq y$ then $x \to y=1$, for any $x, y \in A$, as $x \leq y$ implies $x=x \wedge y$ which by the axiom $(2)$ implies $x \to y=(x \wedge y) \to y=1$. Thirdly, notice that $\nabla$ is order-preserving, as if $x \leq y$ then $x \wedge y=x$. Hence, by axiom $(3)$, we have $\nabla x=\nabla (x \wedge y) \leq \nabla x \wedge \nabla y$, which implies $\nabla x \leq \nabla y$. Now, to prove the adjunction condition of a $\nabla$-algebra, if $\nabla c \wedge a \leq b$ then $(\nabla c \wedge a) \to b=1$ and hence, by the axiom $(5)$, we have $c \leq a \rightarrow b$. Conversely, if $c \leq a \rightarrow b$, then since $\nabla$ is order-preserving, $\nabla c \leq \nabla (a \rightarrow b)$. Therefore, $\nabla c \wedge a \leq \nabla (a \to b) \wedge a$. Finally, by the axiom $(4)$, we reach $\nabla c \wedge a \leq b$. This completes the first part of the claim. For the second part, it is enough to note that in the presence of the lattice axioms, the inequalities $(2)-(5)$ can be rewritten as some equalities. Hence, $\mathcal{V}$ is a variety.

For the third part, it is enough to axiomatize any property in the set $ \{D, N, R, L, Fa, Fu\}$ by an equality. The only non-trivial cases are for $(Fa)$ and $(Fu)$ for which one can use Theorem \ref{Surjectivity} and Theorem \ref{Injectivity}. Hence, for any $C \subseteq \{D, N, R, L, Fa, Fu\}$, the class $\mathcal{V}(C)$  is a variety. The case for $\mathcal{V}_H(C)$ is also easy, using the fact that being a Heyting algebra is definable by equalities using a new symbol for the Heyting implication. 
\end{proof}

\begin{remark}
Note that Theorem \ref{Variety} excludes the classes $\mathcal{V}(C, H)$, for any $C \subseteq \{D, N, R, L, Fa, Fu\}$. The reason is the lack of the Heyting implication in the signature of the Heyting $\nabla$-algebras. This operation is usually required in the axiomatizations that only use the equalities. For instance, one can show that $\mathcal{V}(H)$ is not a variety. The reason simply is the existence of $\nabla$-algebras $\mathcal{A}$ and $\mathcal{B}$ and a $\nabla$-algebra embedding $i: \mathcal{A} \to \mathcal{B}$ such that $\mathcal{B}$ is Heyting while $\mathcal{A}$ is not. To have a concrete example, set $\mathcal{B}=(P(\mathbb{N}), \nabla, \to)$, where $P(\mathbb{N})$ is a boolean algebra of all subsets of $\mathbb{N}$, $\nabla U = \varnothing$ and $U \to V=\mathbb{N}$, for any $U , V \in P(\mathbb{N})$, $\mathcal{A}=(F(\mathbb{N}) \cup \{\mathbb{N}\}, \nabla, \to)$, where $F(\mathbb{N})$ is the set of all finite subsets of $\mathbb{N}$ and $\nabla$ and $\to$ are as in $\mathcal{B}$ and $i: \mathcal{A} \to \mathcal{B}$ is simply the inclusion function. Having said that, it is also worth mentioning that there are some cases where the conditions in $C$ imply the existence of the Heyting structure and hence $\mathcal{V}(C, H)=\mathcal{V}(C)$ is a variety. For instance, consider the cases where $Fa \in C$ and note that any faithful $\nabla$-algebra is Heyting by Corollary \ref{FaithfulImpliesHeyting}.
\end{remark}
\subsection{Abstract Implications}
In this subsection, we will recall the abstract implications and their representations by the $\nabla$-algebras. 
\begin{definition}
Let $\mathsf{A}=(A, \wedge, 1)$ be a bounded meet semi-lattice and $\leq$ be its meet-semilattice order. A binary order-preserving operation $\to$ from $(A, \leq)^{op} \times (A, \leq)$ to $(A, \leq)$ is called an \emph{implication}, if for any $a, b, c \in A$:
\begin{itemize}
\item[$\bullet$]
\emph{(internal reflexivity)} $a \to a=1$,
\item[$\bullet$]
\emph{(internal transitivity)} $(a \to b) \wedge (b \to c) \leq a \to c$.
\end{itemize}
The pair $\mathcal{A}=(\mathsf{A}, \to)$ is called a \emph{strong algebra} (short for strong enough to internalize its order). An implication is called \emph{meet-internalizing} if $a \to (b \wedge c)=(a \to b) \wedge (a \to c)$ and \emph{join-internalizing} (when $\mathsf{A}$ is also a bounded lattice) if $(a \vee b) \to c=(a \to c) \wedge (b \to c)$, for any $a, b, c \in A$. 
In any strong algebra, $\Box a$ is an abbreviation for $1 \to a$. If $\mathcal{A}=(\mathsf{A}, \to_{\mathcal{A}})$ and $\mathcal{B}=(\mathsf{B}, \to_{\mathcal{B}})$ are two strong algebras, by a \emph{strong algebra morphism}, we mean a map $f: A \to B$, preserving the order, all the finite meets and the implication. It is called an \emph{embedding} if it is also an order embedding.
\end{definition}

The following theorem ensures that any $\nabla$-algebra provides an implication while such well-behaved implications are general enough to represent all abstract implications. A weaker version of the second part first appeared in the unpublished draft of the present paper and then generalized in \cite{ImSpace} to its current form.

\begin{theorem}\label{NablaImpliesStrong}
\begin{description}
\item[$(i)$]
Let $(\mathsf{A}, \nabla, \to)$ be a $\nabla$-algebra. Then, $ \to$ is a meet-internalizing implication and if $\mathsf{A}$ is distributive, it is join-internalizing, as well.
\item[$(ii)$]
Let $(\mathsf{A}, \to_{\mathcal{A}})$ be a strong algebra. Then, there exists a locale $\mathscr{X}$, a $\nabla$-algebra $(\mathscr{X}, \nabla, \to_{\mathscr{X}})$, an order-preserving map $F: \mathscr{X} \to \mathscr{X}$ and a bounded meet semi-lattice embedding $i: \mathsf{A} \to \mathscr{X}$ such that 
\[
i(a \to_{\mathcal{A}} b)=F(i(a)) \to_{\mathscr{X}} F(i(b)),
\]
for any $a, b \in A$. Moreover, if $\to_{\mathcal{A}}$ is meet-internalizing, then $F$ can be set as the identity function and hence $i$ becomes a strong algebra morphism. If $\mathsf{A}$ is also distributive and $\to_{\mathcal{A}}$ is join-internalizing, then $i$ can be set as a bounded lattice morphism. 
\end{description}
\end{theorem}
\begin{proof}
For $(i)$, first, notice that if $a \leq b$, then $\nabla 1 \wedge a \leq b$ which by adjunction implies $1 \leq a \to b$ and hence $a \to b=1$. Now, as a special case, since $a \leq a$, we have $a \to a=1$ which proves the internal reflexivity.
For the internal transitivity, let $d \in A$ be an arbitrary element satisfying $d \leq (a \to b) \wedge (b \to c)$. Therefore, $d \leq a \to b$ and $d \leq b \to c$. By the adjunction, $\nabla d \wedge a \leq b$ and $\nabla d \wedge b \leq c$. Hence, $\nabla d \wedge a \leq c$. By the adjunction again, $d \leq a \to c$. As $d$ is arbitrary, we have $(a \to b) \wedge (b \to c) \leq (a \to c)$. To show that $\to$ is increasing (resp. decreasing) in its succedent (antecedent), let
$a' \leq a$ and $b \leq b'$. Then, by the first observation, $(a' \to a)=(b \to b')=1$. Therefore, by the internal transitivity, we have 
\[
(a \to b)=(a' \to a) \wedge (a \to b) \wedge (b \to b') \leq (a' \to b').
\]
To show that $\to$ is meet-internalizing, note that for any $a \in A$, the operation $a \to (-)$ has a left adjoint and hence preserves all finite meets. To prove that $\to$ is join-internalizing, for any arbitrary $d \in A$, consider the following equivalences:
\[
d \leq (a \vee b) \to c \quad \text{iff} \quad \nabla d \wedge (a \vee b) \leq c \quad \text{iff} \quad  (\nabla d \wedge a) \vee (\nabla d \wedge b) \leq c \quad \text{iff}
\]
\[
[\nabla d \wedge a \leq c \;\; \& \;\; \nabla d \wedge b \leq c]
\quad \text{iff} \quad
[d \leq a \to c \;\; \& \;\; d \leq b \to c] \quad 
\text{iff} \quad 
\]
\[
d \leq (a \to c) \wedge (b \to c).
\]
Therefore, as $d$ is arbitrary, we reach $(a \vee b) \to c=(a \to c) \wedge (b \to c)$.

For the first part of $(ii)$, we refer the reader to Theorem 8 in \cite{ImSpace} that proves a generalized version of our claim for monoidal posets rather than meet semi-lattices. The proof also works verbatim here. For the second part, the reader is referred to Theorem 17 in \cite{ImSpace}. 
\end{proof}

To complete the relationship between the strong algebras and the $\nabla$-algebras, notice that Theorem \ref{NablaImpliesStrong} just ensures that any meet-internalizing abstract implication is the implication of a  $\nabla$-algebra, if we enlarge the bounded meet-semilattice base to a locale. However, without the enlargement part, the claim may not hold, even if the base is already a locale. In the following, working over locales, we will provide a characterization for the implications that come from the $\nabla$-algebras.
\begin{theorem}\label{StrongVsNabla}
Let $\mathscr{X}$ be a locale and $(\mathscr{X}, \to)$ be a strong algebra. Then, the following are equivalent:
\begin{description}
\item[$(i)$]
There is an operation $\nabla: \mathscr{X} \to \mathscr{X}$ such that $(\mathscr{X}, \nabla, \to)$ is a $\nabla$-algebra.
\item[$(ii)$]
The operation $\to$ internalizes both the arbitrary meets and the Heyting implication, where the former means $a \to \bigwedge_{i\in I} b_i=\bigwedge_{i\in I} (a \to b_i)$, for any $a \in \mathscr{X}$ and any family $\{b_i\}_{i \in I}$ of the elements of $\mathscr{X}$ and the latter means $a \to (b \supset c)=a \wedge b \to c$, for any $a, b, c \in \mathscr{X}$, where $\supset$ is the Heyting implication of $\mathscr{X}$.
\item[$(iii)$]
The operation $\Box$ preserves all arbitrary meets and for any $b, c \in \mathscr{X}$, we have $b \to c=\Box(b \supset c)$, where $\supset$ is the Heyting implication of $\mathscr{X}$.
\end{description}
\end{theorem}
\begin{proof}
To prove $(ii)$ from $(i)$, since $\nabla(-) \wedge a \dashv a \to (-)$, the operation $a \to (-)$ must preserve all arbitrary meets. Moreover, for any $d \in \mathscr{X}$, we have the following equivalences: $d \leq a \to (b \supset c)$ iff $\nabla d \wedge a \leq b \supset c$ iff $\nabla d \wedge a \wedge b \leq c$ iff $d \leq a \wedge b \to c$. Therefore, $a \to (b \supset c)=a \wedge b \to c$. To prove $(iii)$ from $(ii)$, just set $a=1$ in both equations. Finally, to prove $(i)$ from $(iii)$, since $\Box$ preserves all meets and $\mathscr{X}$ is complete, by the adjoint functor theorem, Theorem \ref{AdjointFunctorTheorem}, $\Box$ has a left adjoint. Call it $\nabla$. As
\[
\nabla d \wedge b \leq c \;\;\; \text{iff} \;\;\; \nabla d \leq b \supset c \;\;\; \text{iff} \;\;\; d \leq \Box (b \supset c) \;\;\; \text{iff} \;\;\; d \leq b \to c,
\]
it is clear that this $\nabla$ works.
\end{proof}
In the following, we will provide an example of a locale $\mathscr{X}$ and a join- and meet-internalizing implication $\to$ such that $a \to (-)$ preserves all the meets, for any $a \in \mathscr{X}$, while it does not internalize the Heyting implication and hence is not a part of a $\nabla$-algebra.
\begin{example}
Consider the set $\{a, b\}$ and define $\leq$ as the partial order generated by the inequality $a \leq b$. Define $f: \{a, b\} \to \{a, b\}$ as $f(a)=f(b)=a$ and set $X$ as the topological space of the upsets of $\{a, b\}$. Hence, $\mathcal{O}(X)$ looks like:
\[\begin{tikzcd}
	{\{a,b\}} \\
	{\{b\}} \\
	\varnothing
	\arrow[no head, from=3-1, to=2-1]
	\arrow[no head, from=2-1, to=1-1]
\end{tikzcd}\]
\normalsize
As $f$ is order-preserving, it is also continuous. Therefore, $f^{-1}$ maps $\mathcal{O}(X)$ into $\mathcal{O}(X)$.
Define $U \to V=f^{-1}(U) \supset f^{-1}(V)$, where $\supset$ is the Heyting implication of $\mathcal{O}(X)$. As $f^{-1}$ is order-preserving, it is easy to check that $\to$ is an implication on $\mathcal{O}(X)$ and as $f^{-1}$ preserves the arbitrary unions and intersections, it is also possible to show that $\to$ is both meet- and join-internalizing. Moreover, as $X$ is finite, any meet is a finite meet. Therefore, the operation $U \to (-)$ is meet-preserving, for any $U \in \mathcal{O}(X)$. However, $\to$ does not internalize the Heyting implication and hence there is no $\nabla$ to make $(\mathcal{O}(X), \nabla, \to)$ a $\nabla$-algebra. Assuming otherwise, we must have $f^{-1}(U) \supset f^{-1}(V)=U \to V=\Box(U \supset V)=f^{-1}(U \supset V)$, for any opens $U$ and $V$ of $X$.
Set $U=\{b\}$ and $V=\varnothing$. Since $f^{-1}(\{b\})=\varnothing$, we have $f^{-1}(\{b\}) \supset f^{-1}(\varnothing)=X$, while $(\{b\} \supset \varnothing)=\varnothing$ and hence $f^{-1}(\{b\} \supset \varnothing)=\varnothing$. Therefore,  $f^{-1}(U \supset V) \neq f^{-1}(U) \supset f^{-1}(V)$ which is a contradiction. 
\end{example}

\section{Subdirectly irreducible and Simple Normal Distributive $\nabla$-algebras}\label{SimpleAlg}

Let us recall some basic definitions from universal algebra specified to the setting of $\nabla$-algebras. Let $\mathcal{A}=(\mathsf{A}, \nabla, \to)$ be a $\nabla$-algebra. A binary relation $\theta \subseteq A \times A$ is called a \emph{congruence relation} if it is an equivalence relation compatible with all the algebraic operations in the signature, namely the finite meets, the finite joins, $\nabla$ and $\to$. We denote the set of all congruence relations of $\mathcal{A}$ by $\Theta(\mathcal{A})$. Any $\nabla$-algebra has two trivial congruence relations, namely the equality and the whole set $A \times A$. A $\nabla$-algebra is called \emph{simple}, if it has no nontrivial congruence relation and it is called \emph{subdirectly irreducible}, if it is either trivial (with exactly one element) or it has a least non-identity congruence with respect to the inclusion. For more information on the general universal algebraic notions, see \cite{Sank} and for the characterization of subdirectly irreducible and simple Heyting algebras, see \cite{Esakia}.

\begin{remark}
Working with normal explicitly Heyting algebras, the previous definition of congruence relations may appear somewhat ambiguous, as it is not clear if a congruence relation must also respect the Heyting implications. To make the definition more clear, let us emphasize that we do not assume this preservation condition. However, we will see that it automatically follows from the original definition. This makes the definition natural, even in its extended signature.
\end{remark}

To provide a characterization for simple and subdirectly irreducible normal distributive $\nabla$-algebras, as it is well-known in the theory of Heyting algebras, we establish a connection between the congruence relations and a family of filters. In our case, the filters we are interested in are what we call the \emph{modal filters}.

\begin{definition}
Let $\mathcal{A}$ be a normal $\nabla$-algebra. By a \emph{modal filter} $F$ on $\mathcal{A}$, we mean an upset of $\mathcal{A}$, closed under all finite meets and the modal operations $\Box$ and $\nabla$. We denote the class of all modal filters of $\mathcal{A}$ by $\mathcal{M}(\mathcal{A})$. 
\end{definition}

\begin{example}
Any normal $\nabla$-algebra has two trivial modal filters $F=\{1\}$ and $F=A$. The latter is clear. For the former, notice that the normality implies $\nabla 1=1$, while $\Box 1=1$ holds in all $\nabla$-algebras.
\end{example}

The following lemma proves the existence and also provides an explicit description for the least modal filter containing a given set $S$ of the elements of a normal $\nabla$-algebra.

\begin{lemma}\label{GenModFil}
Let $\mathcal{A}$ be a normal $\nabla$-algebra. For any subset $S \subseteq A$, the least modal filter extending the set $S$, denoted by $m(S)$, exists and is 
described as:
\[
m(S)=\{y \in A \mid \exists I \, \forall i \in I \, \exists m_i, n_i \in \mathbb{N} \; \exists s_i \in S \; (\bigwedge_{i \in I} \nabla^{m_i} \Box^{n_i} s_i \leq y)\},
\]
where $I$ ranges over the \emph{finite} sets. We will denote $m(\{x\})$ by $m(x)$.
\end{lemma}
\begin{proof}
It is clear that any modal filter that extends $S$ includes $m(S)$. Therefore, it is enough to show that $m(S)$ is a modal filter itself. It is clearly a filter and since $\nabla$ commutes with all finite meets, it is also  closed under $\nabla$. To show its closure under $\Box$, assume $y \in m(S)$. Hence, there is a finite set $I$ such that for any $i \in I$, there are $m_i, n_i \in \mathbb{N}$ and $s_i \in S$ such that $\bigwedge_{i\in I} \nabla^{m_i} \Box^{n_i} s_i \leq y$. Therefore, $\Box \bigwedge_{i\in I} \nabla^{m_i} \Box^{n_i} s_i \leq \Box y$. Since $\Box$ is a right adjoint, it commutes with meets and hence $ \bigwedge_{i \in I} \Box \nabla^{m_i} \Box^{n_i} s_i \leq \Box y$. Define $J$ as the set of all $i$'s in $I$ such that $m_i> 0$. Then, since for any $x \in A$, we have $x \leq \Box \nabla x$, we have $ \nabla^{m_i-1} \Box^{n_i} s_i \leq \Box \nabla^{m_i} \Box^{n_i} s_i $, for any $i \in J$. Hence, 
\[
\left( \bigwedge_{i \in I-J} \Box^{n_i+1} s_i \right) \wedge \left( \bigwedge_{i \in J}  \nabla^{m_i-1} \Box^{n_i} s_i \right) \leq \left( \bigwedge_{i \in I-J} \Box^{n_i+1} s_i \right) \wedge \left( \bigwedge_{i \in J}  \Box \nabla^{m_i} \Box^{n_i} s_i \right)
\]
As the right hand side equals to $\bigwedge_{i \in I} \Box \nabla^{m_i} \Box^{n_i} s_i$ and we observed that it is smaller than $\Box y$, we have $\Box y \in m(S)$.
\end{proof}
In the following, we will show that the modal filters and the congruence relations in normal distributive $\nabla$-algebras are in one-to-one correspondence. To establish that connection, we need the following lemma:
\begin{lemma}\label{InternalCong}
In any normal distributive $\nabla$-algebra, the following inequalities are satisfied:
\begin{description}
\item[$(1)$]
$x \to y \leq (x \wedge z) \to (y \wedge z)$,
\item[$(2)$]
$x \to y \leq (x \vee z) \to (y \vee z)$,
\item[$(3)$]
$\nabla (x \to y) \leq \nabla x \to \nabla y$,
\item[$(4)$]
$\Box (x \to y) \leq (z \to x) \to (z \to y)$,
\item[$(5)$]
$\Box (x \to y) \leq (y \to z) \to (x \to z)$.
\end{description}
If the $\nabla$-algebra is also Heyting with the Heyting implication $\supset$, we also have:
\begin{description}
\item[$(6)$]
$x \to y \leq (z \supset x) \to (z \supset y)$,
\item[$(7)$]
$x \to y \leq (y \supset z) \to (x \supset z)$.
\end{description}
\end{lemma}
\begin{proof}
For $(1)$, using the adjunction, it is enough to prove $\nabla (x \to y) \wedge x \wedge z \leq y \wedge z$, which is clear from $\nabla (x \to y) \wedge x \leq y$. For $(2)$, using the adjunction, we have to show that $\nabla (x \to y) \wedge (x \vee z) \leq y \vee z$. Using distributivity and the fact that $\nabla (x \to y) \wedge x \leq y$, the claim easily follows. For $(3)$, by normality, we have 
\[
\nabla \nabla (x \to y) \wedge \nabla x=\nabla [\nabla (x \to y) \wedge x] \leq \nabla y.
\]
Hence, by adjunction, $\nabla (x \to y) \leq \nabla x \to \nabla y$. For $(4)$, by the adjunction, we have to show 
\[
\nabla \Box (x \to y) \wedge (z \to x) \leq (z \to y).
\]
By Theorem \ref{BasicPropOfNablaAlg}, part (iii), we have $\nabla \Box (x \to y) \leq (x \to y)$. Therefore, using the fact that the operation $\to$ is an implication, Theorem \ref{NablaImpliesStrong}, the claim easily follows. The proof for $(5)$ is similar to that of $(4)$. The parts $(6)$ and $(7)$ are some easy consequences of $\nabla (x \to y) \wedge x \leq y$.
\end{proof}

\begin{theorem}\label{FilWithConCor}
Let $\mathcal{A}$ be a normal distributive $\nabla$-algebra. Then, the following operations provide a one-to-one correspondence between the poset of the modal filters of $\mathcal{A}$ and the poset of the congruences on $\mathcal{A}$, both ordered by the inclusion:
\begin{itemize}
\item[]
$\alpha : \mathcal{M}(\mathcal{A}) \to \Theta(\mathcal{A})$, defined by $\alpha (F)=\{(x, y) \in A^2 \mid x \leftrightarrow y \in F \}$, 
\item[]
$\beta : \Theta (\mathcal{A}) \to \mathcal{M}(\mathcal{A})$, defined by $\beta (\theta)=\{x \in A \mid (x, 1) \in \theta \}$.
\end{itemize}
where $x \leftrightarrow y$ is an abbreviation for $(x \to y) \wedge (y \to x)$.
\end{theorem}
\begin{proof}
First, we show that $\alpha(F)$ is a congruence relation and  $\beta(\theta)$ is a modal filter, for any modal filter $F$ and any congruence relation $\theta$. 
For the former, by Theorem \ref{NablaImpliesStrong}, $\to$ is an implication. Hence, we have $x \to x=1$ and $(x \to y) \wedge (y \to z) \leq (x \to z)$, for any $x, y, z \in A$. Using this fact and the symmetric definition of $\alpha(F)$, it is easy to prove that $\alpha(F)$ is an equivalence relation. To prove that $\alpha(F)$ respects all the operations in the signature, it is enough to use Lemma \ref{InternalCong} and the fact that $F$ is a filter closed under $\Box$ and $\nabla$. We only prove the hardest case of implication, the rest are similar. To show that $\alpha(F)$ respects the operation $\to$, we prove that $(x, y), (z, w) \in \alpha(F)$ imply $(x \to z, y \to w) \in \alpha(F)$.
By definition, as $(x, y), (z, w) \in \alpha(F)$, we have $x \leftrightarrow y, z \leftrightarrow w \in F$. Since $F$ is a filter, $y \rightarrow x, z \rightarrow w \in F$. Since $F$ is a modal filter, $\Box (y \rightarrow x), \Box (z \rightarrow w) \in F$. By Lemma \ref{InternalCong}, parts $(4)$ and $(5)$, we have
\[
\Box (y \rightarrow x) \wedge \Box (z \rightarrow w) \leq [(x \to z) \to (y \to z)] \wedge [(y \to z) \to (y \to w)].
\] 
Since $\to$ is an implication, we have
\[
[(x \to z) \to (y \to z)] \wedge [(y \to z) \to (y \to w)] \leq (x \to z) \to (y \to w),
\]
and as $F$ is a filter, we have $(x \to z) \to (y \to w) \in F$. Similarly, $(y \to w) \to (x \to z) \in F$. Hence, $(x \to z, y \to w) \in \alpha(F)$.\\
To prove that $\beta(\theta)$ is a modal filter, the only thing to check is the upward-closedness of $\beta (\theta)$. The rest is a clear consequence of the equalities $1 \wedge 1=\nabla 1=\Box 1=1$. Now, assume $x \leq y$ and $x \in \beta (\theta)$. Since $x \leq y$, we have $x \to y=1$. Therefore, since $\theta$ is a congruence relation and $(x, 1) \in \theta$, we have $(1 \to y, x \to y) \in \theta$. Hence, $(1 \to y, 1) \in \theta$ and then $(\nabla (1 \to y),\nabla 1) \in \theta$. By disjunction with $y$ on both sides, we have $(\nabla (1 \to y) \vee y,\nabla 1 \vee y) \in \theta$. Since $\nabla (1 \to y) \leq y$ and $\nabla 1=1$, we have $(y, 1) \in \theta$. Therefore, $y \in \beta(\theta)$.\\
Notice that both $\alpha$ and $\beta$ are trivially monotone with respect to the inclusion.
To prove $\alpha (\beta (\theta))=\theta$ and $\beta (\alpha (F))=F$, we have to show that 
\[
x \in F \;\;\; \text{iff} \;\;\; 1 \leftrightarrow x \in F \;\;\;\;\;\;\;\;\;\;\;\;\;\;\;\;  (x, y) \in \theta \;\;\; \text{iff} \;\;\; (x \leftrightarrow y, 1) \in \theta.
\]
For the left equivalence, if $x \in F$, since $x \to 1=1$, we have $x \to 1 \in F$ and since $F$ is a modal filter, we have $\Box x=1 \to x \in F$. Therefore, $1 \leftrightarrow x \in F$. For the converse, if $1 \leftrightarrow x \in F$, then $1 \to x \in F$ and as $F$ is upward-closed and also closed under $\nabla$, using the fact that $\nabla (1 \to x) \leq x$, we have $x \in F$.\\
For the right equivalence, if $(x, y) \in \theta$, since $\theta$ is a congruence relation, we have $(x \leftrightarrow y, x \leftrightarrow x) \in \theta$ which implies $(x \leftrightarrow y, 1) \in \theta$, by the fact that $\to$ is an implication and hence $x \to x=1$. For the converse, if $(x \leftrightarrow y, 1) \in \theta$, then since $\beta (\theta)$ is a filter, we have $(x \to y, 1) \in \theta$ and hence $(y\vee[x \wedge \nabla (x \to y)], y \vee [x \wedge \nabla 1]) \in \theta$. Therefore, as $x \wedge \nabla (x \to y) \leq y$ and $\nabla 1=1$, we have $(y, y \vee x) \in \theta$. By symmetry, we also have $(x, x \vee y) \in \theta$ and hence $(x,  y) \in \theta$.
\end{proof}

\begin{remark}\label{CongruenceHeyting}
Theorem \ref{FilWithConCor} shows that if $\mathcal{A}$ is also a Heyting algebra and $\theta$ is a congruence relation, then $\theta$ also respects the Heyting implication. To prove this, assume $(x, y), (z, w) \in \theta$. Set $F=\beta(\theta)$. By Theorem \ref{FilWithConCor}, we have $\theta=\alpha(F)$. Hence, $x \leftrightarrow y, z \leftrightarrow w \in F$. Then, by the parts $(6)$ and $(7)$ in Lemma  \ref{InternalCong}, we have
\[
(y \rightarrow x) \wedge (z \rightarrow w) \leq [(x \supset z) \to (y \supset z)] \wedge [(y \supset z) \to (y \supset w)].
\] 
Since $\to$ is an implication, we have
\[
[(x \supset z) \to (y \supset z)] \wedge [(y \supset z) \to (y \supset w)] \leq (x \supset z) \to (y \supset w).
\]
Hence, we have $(x \supset z) \to (y \supset w) \in F$. Similarly, $(y \supset w) \to (x \supset z)  \in F$. Hence, $(x \supset z) \leftrightarrow (y \supset w) \in F$ which implies $(x \supset z, y \supset w) \in \theta$. This remark ensures that all the following theorems on congruence extension property, subdirectly irreducible and simple normal distributive $\nabla$-algebras also hold for the normal explicitly Heyting $\nabla$-algebras, where the Heyting implication is explicitly mentioned in the signature of the algebras.
\end{remark}

\begin{remark}\label{EquivalenceBetweenTrivialities}
By Theorem \ref{FilWithConCor}, the trivial modal filters $F=\{1\}$ and $F=A$ correspond to the trivial congruence relations $\theta=\{(x, x) \in A^2 \mid x \in A\}$ and $\theta=A \times A$, respectively.
\end{remark}

\begin{definition}
A class $\mathfrak{C}$ of $\nabla$-algebras has the \emph{congruence extension property} if for any $\mathcal{B}$ in $\mathfrak{C}$, any sub-algebra $\mathcal{A}$ of $\mathcal{B}$ and any congruence relation $\theta$ over $\mathcal{A}$, there exists a congruence relation $\phi$ over $\mathcal{B}$ such that $\phi|_{A}=\theta$. 
\end{definition}

\begin{corollary}
The variety $\mathcal{V}(D, N)$ and all of its subclasses have the congruence extension property. The same also holds for $\mathcal{V}_H(N)$.
\end{corollary}
\begin{proof}
For the first part, let $\mathcal{A}$ and $\mathcal{B}$ be two normal distributive $\nabla$-algebras, $\mathcal{A}$ be a sub-algebra of $\mathcal{B}$ and $\theta$ be a congruence relation over $\mathcal{A}$. Define $\phi$ on $\mathcal{B}$ by $\{(x, y) \in B^2 \mid x \leftrightarrow y \in m_{\mathcal{B}}( \beta(\theta)) \}$, where $\beta(\theta)$ is the corresponding modal filter to $\theta$ over $\mathcal{A}$ and $m_{\mathcal{B}}( \beta(\theta))$ is the least modal filter in $\mathcal{B}$ that includes $\beta(\theta)$. By Theorem \ref{FilWithConCor}, $\phi$ is a congruence relation over $\mathcal{B}$. Hence, the only thing remained to prove is that $\phi|_{A}=\theta$. Let $a, b \in A$. Then, we have to show that $(a, b) \in \phi $ iff $(a, b) \in \theta$. Using Theorem \ref{FilWithConCor} for $\mathcal{A}$, the latter is equivalent to $a \leftrightarrow b \in \beta(\theta)$. Now, it is enough to show that $a \leftrightarrow b \in m_{\mathcal{B}}( \beta(\theta))$ iff $a \leftrightarrow b \in \beta (\theta)$. The latter is actually true in a more general form: $c \in m_{\mathcal{B}}(F)$ iff $c \in F$, for any $c \in A$ and any modal filter $F$ over $\mathcal{A}$. One direction is clear. For the other direction, note that if $c \in m_{\mathcal{B}}(F)$, then by Lemma \ref{GenModFil}, there exists a finite set $I$ such that for any $i \in I$, there are $m_i, n_i \in \mathbb{N}$ and $a_i \in F$ such that $\bigwedge_{i \in I} \nabla^{m_i} \Box^{n_i} a_i \leq_{\mathcal{B}} c$. Since, $\mathcal{A}$ is a sub-algebra of $\mathcal{B}$, we have $\bigwedge_{i \in I} \nabla^{m_i} \Box^{n_i} a_i \leq_{\mathcal{A}} c$ which finally implies $c \in F$. This completes the proof of the first part.
For the second part, use the first part and Remark \ref{CongruenceHeyting}.
\end{proof}

\begin{corollary}\label{SubIre}
A non-trivial normal distributive $\nabla$-algebra $\mathcal{A}$ is subdirectly irreducible iff there exists $x \in A-\{1\}$ such that for any $y \in A -\{1\}$, there exists a finite set $I$ such that for any $i \in I$, there are $m_i, n_i \in \mathbb{N}$ satisfying $\bigwedge_{i \in I} \nabla^{m_i} \Box^{n_i} y \leq x$. The same also holds for any $\mathcal{A} \in \mathcal{V}_H(N)$.
\end{corollary}
\begin{proof}
For the first part, by Theorem \ref{FilWithConCor} and Remark \ref{EquivalenceBetweenTrivialities}, it is enough to prove the equivalence between the existence of the minimum modal filter satisfying $F \neq \{1\}$ and the existence of an element $x \in A-\{1\}$ as presented. First, assume that such an $x$ exists. Since $x \neq 1$ and $x \in m(x)$, then $m(x) \neq \{1\}$. Therefore, it is enough to show that $m(x) \subseteq F$, for any $F \neq \{1\}$. Since, $F \neq \{1\}$, there is $y \in F$ such that $y \neq 1$. By the condition, there exists a finite set $I$ such that for any $i \in I$, there are $m_i, n_i \in \mathbb{N}$ satisfying $\bigwedge_{i \in I} \nabla^{m_i} \Box^{n_i} y \leq x$. Since, $F$ is a modal filter and $y \in F$, we have $x \in F$. Therefore, as $m(x)$ is the least modal filter containing $x$, we reach $m(x) \subseteq F$. Conversely, assume that the family of the modal filters different from $\{1\}$ has the minimum element. Denote this element by $F$. Then, since $F \neq \{1\}$, there exists $x \in F$ such that $x \neq 1$. Let $y$ be any arbitrary element in $A - \{1\}$. Then, as $y \in m(y)$, we have $m(y) \neq \{1\}$ which by the minimality of $F$ implies $F \subseteq m(y)$. Since $x \in F$, we have $x \in m(y)$. Therefore, by Lemma \ref{GenModFil}, there exists a finite set $I$ such that for any $i \in I$, there are $m_i, n_i \in \mathbb{N}$ satisfying $\bigwedge_{i \in I} \nabla^{m_i} \Box^{n_i} y \leq x$. This completes the proof of the first part. For the second part, use the first part and Remark \ref{CongruenceHeyting}.
\end{proof}

\begin{corollary}\label{SubDirLeft}
\begin{description}
\item[$(i)$]
A non-trivial normal distributive left $\nabla$-algebra $\mathcal{A}$ is subdirectly irreducible iff there exists $x \in A-\{1\}$ such that for any $y \in A -\{1\}$, there exists $k \in \mathbb{N}$ satisfying $\nabla^{k} y \leq x$. The same also holds for any $\mathcal{A} \in \mathcal{V}_H(N, L)$.
\item[$(ii)$]
A non-trivial normal distributive right $\nabla$-algebra $\mathcal{A}$ is subdirectly irreducible iff there exists $x \in A-\{1\}$ such that for any $y \in A -\{1\}$, there exists $k \in \mathbb{N}$ satisfying $\Box^{k} y \leq x$. The same also holds for any $\mathcal{A} \in \mathcal{V}_H(N, R)$.
\end{description}
\end{corollary}
\begin{proof}
For $(i)$, by Theorem \ref{Left}, we have $z \leq \Box z$ and $\nabla^{l+1} z \leq \nabla^l z$, for any $z \in A$. Therefore, for any finite set $I$, any $m_i, n_i \in \mathbb{N}$ and any $y \in A$, we have $\nabla^k y \leq \bigwedge_{i \in I} \nabla^{m_i} \Box^{n_i} y$, where $k$ is the maximum of all $m_i$'s. Using this observation, it is easy to see that the condition in Corollary \ref{SubIre} is equivalent to the existence of $x \neq 1$ such that for any $y \neq 1$, there exists a number $k \in \mathbb{N}$ satisfying $\nabla^{k} y \leq x$. 
For $(ii)$, use a similar argument, considering the fact that in any right $\nabla$-algebra, $z \leq \nabla z$ and $\Box^{l+1} z \leq \Box^l z$, for any $z \in A$, as proved in Theorem \ref{Right}.
\end{proof}

\begin{corollary}
A non-trivial Heyting algebra $\mathcal{A}$ is subdirectly irreducible iff there exists $x \in A-\{1\}$ such that $ y \leq x$, for any $y \in A -\{1\}$. 
\end{corollary}
\begin{proof}
A Heyting algebra is a normal distributive $\nabla$-algebra with $\nabla=\Box=id$. Now, apply Corollary \ref{SubIre}.
\end{proof}

\begin{corollary}\label{Simple}
A normal distributive $\nabla$-algebra $\mathcal{A}$ is simple iff for any $x \in A-\{1\}$, there exists a finite set $I$ such that for any $i \in I$, there are $m_i, n_i \in \mathbb{N}$ satisfying $\bigwedge_{i \in I} \nabla^{m_i} \Box^{n_i} x=0$. The same also holds for any $\mathcal{A} \in \mathcal{V}_H(N)$.
\end{corollary}
\begin{proof}
For the first part, by Theorem \ref{FilWithConCor} and Remark \ref{EquivalenceBetweenTrivialities}, it is enough to show that the condition in the statement of the corollary is equivalent to the non-existence of non-trivial modal filters. First, assume the condition and suppose that $F \neq \{1\}$ is an arbitrary modal filter. We have to show that $F= A$. Since $F \neq \{1\}$, there is $x \in F$ such that $x \neq 1$. By the condition, there exists a finite set $I$ such that for any $i \in I$, there are $m_i, n_i \in \mathbb{N}$ satisfying $\bigwedge_{i \in I} \nabla^{m_i} \Box^{n_i} x =0$. Since $x \in F$, we have $0 \in F$ and hence $F=A$. Conversely, assume that $\mathcal{A}$ has no non-trivial modal filter. Let $x \neq 1$. As $x \in m(x)$, we have $m(x) \neq \{1\}$. Therefore, $m(x)=A$, which implies $0 \in m(x)$. By Lemma \ref{GenModFil}, there exists a finite set $I$ such that for any $i \in I$, there are $m_i, n_i \in \mathbb{N}$ satisfying $\bigwedge_{i \in I} \nabla^{m_i} \Box^{n_i} x=0$. This completes the proof of the first part. For the second part, use the first part and Remark \ref{CongruenceHeyting}.
\end{proof}

\begin{corollary}\label{SimpleForLeftRight}
\begin{description}
\item[$(i)$]
A normal distributive left $\nabla$-algebra $\mathcal{A}$ is simple iff for any $x \in A -\{1\}$, there exists $k \in \mathbb{N}$ satisfying $\nabla^{k} x=0$. The same also holds for any $\mathcal{A} \in \mathcal{V}_H(N, L)$.
\item[$(ii)$]
A normal distributive right $\nabla$-algebra $\mathcal{A}$ is simple iff for any $x \in A -\{1\}$, there exists $k \in \mathbb{N}$ satisfying $\Box^{k} x=0$. The same also holds for any $\mathcal{A} \in \mathcal{V}_H(N, R)$.
\end{description}
\end{corollary}
\begin{proof}
For $(i)$, again use the facts $z \leq \Box z$ and $\nabla^{l+1} z \leq \nabla^l z$, for any $z \in A$, proved in Theorem \ref{Left}. For $(ii)$, use $z \leq \nabla z$ and $\Box^{l+1} z \leq \Box^l z$, for any $z \in A$, proved in Theorem \ref{Right}.
\end{proof}

The following characterization of simple Heyting algebras is a well-known fact and it can be obtained as a particular case of Corollary \ref{SimpleForLeftRight}.

\begin{corollary}
A Heyting algebra $\mathcal{A}$ is simple iff for any $x \in A$ either $x=1$ or $x=0$.
\end{corollary}

The following theorem provides an infinite family of finite simple normal Heyting $\nabla$-algebras to witness the relative complexity of the varieties of $\nabla$-algebras compared to the variety of Heyting algebras, where there is only one non-trivial simple element.
\begin{theorem}
There are infinitely many finite simple normal Heyting $\nabla$-algebras.
\end{theorem}
\begin{proof}
Let $n \geq 1$ be a natural number. Define the topological space $X_n=\{1, 2, \cdots, n\} \cup \{\omega\}$ with the following topology: A subset $U \subseteq X_n$ is open if either $U=X_n$ or $U \subseteq \{1, 2, \cdots, n\}$. Define $f_n: X_n \to X_n$ as the function mapping $x \neq n, \omega$ to $x+1$ and $f_n(n)=f_n(\omega)=\omega$. This function is clearly continuous as $\omega \notin f_n^{-1}(U)$, for any $U \subseteq \{0, 1, \cdots, n\}$. Now, consider the normal Heyting $\nabla$-algebra $\mathcal{A}_n$ associated to the dynamic topological system $(X_n, f_n)$ as explained in Example \ref{TopologicalSpace} and recall that in that $\nabla$-algebra $\nabla U=f^{-1}(U)$, for any open $U \subseteq X_n$. We claim that $\mathcal{A}_n$ is simple. For that purpose, we show
$\nabla^n U=\varnothing$, for any $U \neq X_n$. Assume $U \neq X_n$. Hence, $\omega \notin U$. 
It is easy to see that $\nabla^m U =\{x \in \{1, 2, \cdots, n\} \mid x+m \in U\}$, for any $m \leq n$. Therefore, for $m=n$, we have $\nabla^n U=\varnothing$.
Hence, by Corollary \ref{Simple}, the normal distributive $\nabla$-algebra $\mathcal{A}_n$ is simple. Finally, note that the $\mathcal{A}_n$'s are different as their base lattices have different cardinalities.
\end{proof}

\section{Dedekind-MacNeille Completion}\label{Dedekind}
The Dedekind-MacNeille completion of a poset is its smallest complete extension. 
In this section, we will show that starting from a $\nabla$-algebra rather than a poset, there is a unique way to make its  Dedekind-MacNeille completion a $\nabla$-algebra. In this sense, one can say that the variety of $\nabla$-algebras is closed under the Dedekind-MacNeille completion. This closure will be useful later in the sequel of the present paper to prove a completeness theorem showing that certain logics are complete for certain families of \emph{complete} $\nabla$-algebras. The completion is also the main tool to prove the completeness of the algebraic semantics for the predicate version of the mentioned logics. Note that the closure under the Dedekind-MacNeille completion is a rare property for a variety to have. For instance, the variety of bounded distributive lattices lacks the property and between all the varieties of Heyting algebras only the three varieties of all Heyting algebras, all boolean algebras and the trivial variety has the property \cite{HardBezh}. In this sense, the closure of many varieties of $\nabla$-algebras under the Dedekind-MacNeille completion can be seen as a mathematical evidence to show how well-behaved the notion of a $\nabla$-algebra is. 

\begin{definition}
Let $(P, \leq)$ be a poset and $S \subseteq P$. By $U(S)=\{x \in P \mid \forall y \in S, \; y \leq x\}$ and $L(S)=\{x \in P \mid \forall y \in S, \; x \leq y\}$, we mean the set of all upper bounds and all lower bounds of $S$, respectively. A subset $N \subseteq P$ is called \emph{normal} iff $LU(N)=N$. Any set in the form $(x]=\{y \in P \mid y \leq x\}$ is clearly normal. The Dedekind-MacNeille completion of $(P, \leq)$, denoted by $\mathcal{N}(P, \leq)$, is the poset of all normal subsets of $P$ ordered by inclusion. By the canonical map of the completion, we mean $i: (P, \leq) \to \mathcal{N}(P, \leq)$ defined by $i(x)=(x]$.
\end{definition}

It is well-known that the poset of all normal subsets is actually a complete lattice with intersection as the meet and  $\bigvee_{i \in I}N_i=LU(\bigcup_{i \in I} N_i)$ as the join. Moreover, the canonical map $i$ is an order embedding that preserves all the existing meets and joins in $(P, \leq)$ and if the poset $(P, \leq)$ is also a Heyting algebra, it also preserves the Heyting implication, see Subsection A.6 in \cite{Esakia}. It is also easy to see that if $\mathcal{A}$ is a bounded lattice itself, any normal subset is actually an ideal. Therefore, as we work exclusively with $\nabla$-algebras which are also bounded lattices, we use the term normal ideals for the normal subsets. 
The following well-known lemma is a nice characterization of these normal ideals. We will provide its simple proof for the sake of completeness.

\begin{lemma}\label{MeetDensity}
A subset $N$ is normal iff $N$ is representable as an intersection of principal ideals, i.e., the ideals in the form $(x]$, for some $x \in A$. 
\end{lemma}
\begin{proof}
If $N$ is normal, as $LU(N)=N$ and $LU(N)=\bigcap_{n \in U(N)} (n]$, the claim is clear. Conversely, assume $N=\bigcap_{i \in I} (n_i]$. We show $LU(N) \subseteq N$. The converse, i.e., $N \subseteq LU(N)$, always holds.
Since $N=\bigcap_{i \in I} (n_i]$, we have $n_i \in U(N)$, because if $x \in N$ then $x \leq n_i$. Now, assume $y \in LU(N)$. Since, $n_i \in U(N)$, we have $y \leq n_i$, for all $i \in I$. This implies $y \in \bigcap_{i \in I} (n_i]=N$. 
\end{proof}

Now, for any $\nabla$-algebra $\mathcal{A}$, define the operations $\nabla$ and $\to$ on its lattice of normal ideals, $\mathcal{N}(\mathsf{A})$, by
$\nabla N = \bigvee_{n \in N} (\nabla n]$ and  $M \to N = \{x \in A \mid \forall m \in M, \nabla x \wedge m \in N \}$.
In the rest of this section, we show that this pair of operations is the unique pair that makes the lattice $\mathcal{N}(\mathsf{A})$ into a $\nabla$-algebra and the lattice embedding $i$ into a $\nabla$-algebra morphism. The uniqueness is easy to prove. Let $\nabla'$ and $\to'$ be another such pair. It is enough to prove $\nabla=\nabla'$. As $\nabla'$ has a right adjoint, it must preserve all joins in $\mathcal{N}(\mathsf{A})$. Moreover, as $i$ is a $\nabla$-algebra morphism, we have $\nabla'(n]=(\nabla n]$. Therefore, 
\[
\nabla'N = \nabla' \bigvee_{n \in N} (n]= \bigvee_{n \in N} \nabla'(n]= \bigvee_{n \in N} (\nabla n] = \nabla N.
\]
This completes the proof of the uniqueness. For the rest, denote the tuple $(\mathcal{N}(\mathsf{A}), \nabla, \to)$ by $\mathcal{N}(\mathcal{A})$. Then:

\begin{theorem}\label{DedekindCompletionThm}
Let $C \subseteq \{H, N, R, L, Fa, Fu\}$. Then, $\mathcal{N}(\mathcal{A}) \in \mathcal{V}(C)$ and the canonical embedding $i: \mathcal{A} \to \mathcal{N}(\mathcal{A})$ is also a $\nabla$-algebra embedding, for any $ \mathcal{A} \in \mathcal{V}(C)$. If $H \in C$, the embedding is also a Heyting $\nabla$-algebra morphism. In this sense, $\mathcal{V}(C)$ is closed under the Dedekind-MacNeille completion.
\end{theorem}
\begin{proof}
To show that the operations $\nabla$ and $\to$ are well-defined, the only thing to prove is the normality of $M \to N$, when $M$ and $N$ are both normal ideals. Since $N$ is normal, by Lemma \ref{MeetDensity}, $N=\bigcap_{i \in I} (n_i]$, for some $I$ and $n_i$'s. Therefore, using the definition, we know that $x \in M \to N$ iff for any $m \in M$ and $i \in I$, we have $\nabla x \wedge m \leq n_i$. Hence, $M \to N=\bigcap_{i \in I, m \in M} (m \to n_i]$ and by Lemma \ref{MeetDensity}, $M \to N$ is normal. 

Now, we show that $i : \mathcal{A} \to \mathcal{N}(\mathcal{A})$ preserves $\nabla$ and $\to$. For $\nabla$, by definition,
$\nabla (a]= \bigvee_{b \leq a} (\nabla b]$. Clearly, $(\nabla a]$ is one of the ideals we use join over. Hence, $(\nabla a] \subseteq \bigvee_{b \leq a} (\nabla b]$. On the other hand, $\nabla b \leq \nabla a$, for any $b \leq a$. Hence, $(\nabla b] \subseteq (\nabla a]$ which implies $\bigvee _{b \leq a} (\nabla b] \subseteq (\nabla a]$. Therefore, $\nabla (a]= \bigvee_{b \leq a} (\nabla b]= (\nabla a]$. For the implication, we have the following series of equivalences
\[
x \in (a] \to (b] \;\;\; \text{iff} \;\;\; \forall y \in (a], \nabla x \wedge y \in (b] \;\;\; \text{iff}
\]
\[
\nabla x \wedge a \leq b \;\;\; \text{iff} \;\;\; x \leq a \to b \;\;\; \text{iff} \;\;\; x \in (a \to b].
\]
Therefore, $(a] \to (b]=(a \to b]$. We now prove the adjunction condition of a $\nabla$-algebra, i.e., 
\[
\nabla M \cap N \subseteq K \;\;\; \text{iff} \;\;\; M \subseteq N \to K,
\]
for any normal ideals $M$, $N$ and $K$.
For the left to right direction, assume $\nabla M \cap N \subseteq K$ and $m \in M$. We have to show that $m \in N \to K$ which means that for any arbitrary $n \in N$, we must have $\nabla m \wedge n \in K$. For that purpose, note that $\nabla m \wedge n \in (\nabla m]=\nabla (m] \subseteq \nabla M$ and $\nabla m \wedge n \in N$, and since $\nabla M \cap N \subseteq K$, we have $\nabla m \wedge n \in K$.\\
For the right to left direction, assume that $M \subseteq N \to K$ and $x \in \nabla M \cap N$. Then, we have to show that $x \in K$. Since $K$ is normal, by Lemma \ref{MeetDensity}, it is representable as $\bigcap_{i \in I} (k_i]$. Therefore, it is enough to show that $x \leq k_i$, for any $i \in I$. First, since $M \subseteq N \to K$, for any $m \in M$ and $n \in N$, we have $\nabla m \wedge n \in K$ which implies $\nabla m \wedge n \leq k_i$, for any $i \in I$. Now, as $x \in N$, for any $m \in M$, we have $\nabla m \wedge x \leq k_i$. Hence, $m \leq x \to k_i$ which implies $\nabla m \leq \nabla (x \to k_i)$. Therefore, $\nabla (x \to k_i)$ is an upper bound for $\bigcup_{m \in M}(\nabla m]$. Hence, $\nabla (x \to k_i) \in U(\bigcup_{m \in M}(\nabla m])$. On the other hand, by definition, $x \in \nabla M=\bigvee_{m \in M} (\nabla m]=LU(\bigcup_{m \in M}(\nabla m])$ which implies $x \leq \nabla (x \to k_i)$. Therefore, $x \leq x \wedge \nabla (x \to k_i) \leq k_i$ which is what we wanted.

Finally, we have to check that the conditions $\{H, N, R, L, Fa, Fu\}$ are preserved under the Dedekind-MacNeille completion. The Heyting case and the fact that the embedding $i$ preserves the Heyting implication is well-known \cite{Esakia}. For $(N)$, we first prove $\nabla A=A$. For that purpose, as $\nabla 1=1$, we have $A \subseteq (\nabla 1]$. Now, since $\nabla A=\bigvee_{a \in A} (\nabla a]$, we have $A \subseteq \nabla A$. To prove that $\nabla$ commutes with the binary meet, it is enough to show that $\nabla M \cap \nabla N \subseteq \nabla (M \cap N)$, for any two normal ideals $M$ and $N$. The other direction always holds in a $\nabla$-algebra. Assume $x \in \nabla M \cap \nabla N$. Since $\nabla M=\bigvee_{m \in M} (\nabla m]=LU(\bigcup_{m \in M} (\nabla m])$, we know that ``\emph{for any $y$, if $y \geq \nabla m$, for all $m \in M$, we have $x \leq y$}". Call the property inside the quotation mark the $(*)$ property for $M$. The $(*)$ property is also true for $N$, because we also have $x \in \nabla N$. To prove $x \in \nabla (M \cap N)=LU(\bigcup_{k \in M \cap N} (\nabla k])$, we assume $z \in U(\bigcup_{k \in M \cap N} (\nabla k])$ and we show that $x \leq z$. Since $z \in U(\bigcup_{k \in M \cap N} (\nabla k])$, for any arbitrary $m \in M$ and $n \in N$, we have $z \geq \nabla (m \wedge n)$. Since $\mathcal{A}$ is normal, we have $z \geq \nabla m \wedge \nabla n$. Hence, $m \leq \nabla n \to z$ which implies $\nabla m \leq \nabla (\nabla n \to z)$. Therefore, by the $(*)$ property for $M$, we have $x \leq \nabla (\nabla n \to z)$. Thus, $ \nabla n \wedge x \leq \nabla n \wedge \nabla (\nabla n \to z) \leq z$. Since, $\nabla n \wedge x \leq z$, we have $n \leq x \to z$ and hence  $\nabla n \leq \nabla (x \to z)$. Therefore, by the $(*)$ property for $N$, we have $x \leq \nabla (x \to z)$ which implies $x \leq x \wedge \nabla(x \to z) \leq z$.\\
The cases for $(L)$ and $(R)$ are easy, using the explicit definition of $\Box$ as $\Box N=\{x \in A \mid \nabla x \in N\}$ and Theorem \ref{Left} and Theorem \ref{Right} that rephrase the leftness and rightness of a $\nabla$-algebra in terms of its $\Box$. For $(Fa)$, by Theorem \ref{Surjectivity}, it is enough to show $\nabla \Box N=N$. We only prove $N \subseteq \nabla \Box N$. The other direction is a consequence of the adjunction $\nabla \dashv \Box$. If $n \in N$, then by $(Fa)$ for $\mathcal{A}$, we have $n \leq \nabla \Box n$. Hence, $n \in (\nabla \Box n]=\nabla \Box (n] \subseteq \nabla \Box N$. For $(Fu)$, by Theorem \ref{Injectivity}, it is enough to show $\Box \nabla N=N$. We only prove $\Box \nabla N \subseteq N$. The other direction is again a consequence of the adjunction $\nabla \dashv \Box$. Assume that $x \in \Box \nabla N$. By definition, we have $\nabla x \in \nabla N$. To prove that $x \in LU(N)=N$, it is enough to pick an arbitrary $y \in U(N)$ and show that $x \leq y$. Since $y \in U(N)$,  for any $n \in N$, we have $y \geq n$ which implies  $\nabla y \geq \nabla n$ and hence $\nabla y \in U(\bigcup_{n \in N} (\nabla n])$. Since $\nabla x \in \nabla N=LU(\bigcup_{n \in N} (\nabla n])$, we have $\nabla x \leq \nabla y$. Since $\mathcal{A}$ satisfies $(Fu)$, by Theorem \ref{Injectivity}, the operation $\nabla$ is an order embedding and hence $x \leq y$. Therefore, $x \in LU(N)=N$.
\end{proof}

\section{Kripke Frames} \label{Kripke}

In this section, we will first recall a variant of intuitionistic Kripke frames, employed in \cite{Sim,Servi,LitViss} to capture different intuitionistic modal logics. This variant provides a natural family of $\nabla$-algebras, as explained in \cite{ImSpace} and Example \ref{KripkeFrames}. Using the usual prime filter construction, also employed in \cite{ImSpace}, it is not hard to represent different classes of distributive $\nabla$-algebras by their corresponding classes of Kripke frames. Here, we recall the prime filter construction to expand the characterization of \cite{ImSpace} to also cover the case of full and faithful distributive $\nabla$-algebras and prove the amalgamation property for some varieties of distributive normal $\nabla$-algebras. The machinery is also required for the Kripke and topological completeness theorems and the duality theory for some families of distributive $\nabla$-algebras planned to be covered in the sequel of the present paper.
\begin{definition}
Let $(W, \leq)$ be a poset. The tuple $(W, \leq, R)$ is called a \emph{Kripke frame} if $R$ is compatible with $\leq$, i.e., if $k'\leq k$, $(k, l) \in R$ and $l \leq l'$, then $(k', l') \in R$:
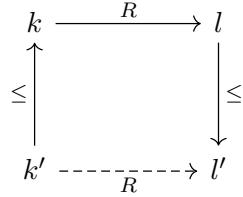
\[\small \begin{tikzcd}
	k && l \\
	\\
	{k'} && {l'}
	\arrow["\leq", from=3-1, to=1-1]
	\arrow["R", from=1-1, to=1-3]
	\arrow["\leq", from=1-3, to=3-3]
	\arrow["R"', dashed, from=3-1, to=3-3]
\end{tikzcd}\]
for any $k,l,k',l'\in W$. Moreover, 
\begin{description}
\item[$(N)$]
if there exists an order-preserving function $\pi : W \to W$, called  the \emph{normality witness}, such that $(x, y) \in R$ iff $x \leq \pi(y)$, then the Kripke frame is called \emph{normal},
\item[$(R)$]
if $R$ is reflexive or equivalently $\leq \, \subseteq R$, then the Kripke frame is called \emph{right},
\item[$(L)$]
if $R \; \subseteq \; \leq$, then the Kripke frame is called \emph{left},
\item[$(Fa)$]
if for any $x \in W$, there exists $y \in W$ such that $(y, x) \in R$ and for any $z \in W$ such that $(y, z) \in R$ we have $x \leq z$, then the Kripke frame is called \emph{faithful},
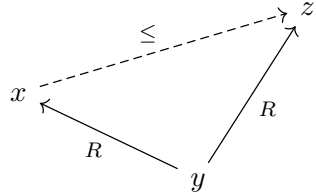
\[\small \begin{tikzcd}
	&&& z \\
	x \\
	&& y
	\arrow["R", from=3-3, to=2-1]
	\arrow["R"', from=3-3, to=1-4]
	\arrow["\leq", dashed, from=2-1, to=1-4]
\end{tikzcd}\]
\item[$(Fu)$]
if for any $x \in W$, there exists $y \in W$ such that $(x, y) \in R$ and for any $z \in W$ such that $(z, y) \in R$ we have $z \leq x$, then the Kripke frame is called \emph{full}.
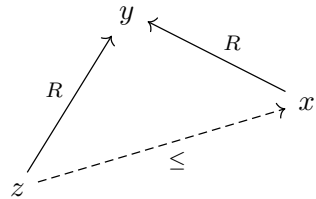
\[\small \begin{tikzcd}
	& y \\
	&&& x \\
	z
	\arrow["R"', from=2-4, to=1-2]
	\arrow["R", from=3-1, to=1-2]
	\arrow["\leq"', dashed, from=3-1, to=2-4]
\end{tikzcd}\]
\end{description}
For any $C \subseteq \{N, R, L, Fa, Fu\}$, by $\mathbf{K}(C)$, we mean the class of all Kripke frames with the properties described in the set $C$. For instance, $\mathbf{K}(\{N, Fa\})$ is the class of all normal faithful Kripke frames. \\
If $\mathcal{K}=(W, \leq, R)$ and $\mathcal{K}'=(W', \leq', R')$ are two Kripke frames, then by a Kripke morphism $f: \mathcal{K} \to \mathcal{K}'$, we mean an order-preserving function from $W$ to $W'$ such that:
\begin{itemize}
\item[$\bullet$]
For any $k, l \in W$, if $(k, l) \in R$ then $(f(k), f(l)) \in R'$,
\item[$\bullet$]
for any $k \in W$ and any $l' \in W'$ such that $(f(k), l') \in R'$, there exists $l \in W$ such that $(k, l) \in R$ and $f(l)=l'$,
\item[$\bullet$]
for any $k \in W$ and any $l' \in W'$ such that $(l', f(k)) \in R'$, there exists $l \in W$ such that $(l, k) \in R$ and $f(l) \geq' l'$.
\end{itemize}
If we also have the following condition:
\begin{itemize}
\item[$\bullet$]
for any $k \in W$ and any $l' \in W'$ such that $f(k) \leq' l'$, there exists $l \in W$ such that $k \leq l$ and $f(l)=l'$,
\end{itemize}
then the Kripke morphism $f$ is called a \emph{Heyting  Kripke morphism}. Kripke frames and Kripke morphisms form a category that we loosely denote by its class of objects $\mathbf{K}(C)$. If we use the Heyting Kripke morphisms, instead, then we denote the subcategory by $\mathbf{K}^H(C)$. 
\end{definition}
If the Kripke frame is normal, we can rewrite the conditions $\{R, L, Fa, Fu\}$ purely in terms of the normality witness of the frame.
We will need these characterizations to prove the amalgamation property for some varieties of normal $\nabla$-algebras in Subsection \ref{SectionAmalgamation}. They are also required in the duality theory planned to be covered in the sequel of the present paper.
\begin{lemma}\label{NormalForConditions}
Let $\mathcal{K}=(W, \leq, R)$ be a normal Kripke frame with the normality witness $\pi$. Then: 
\begin{description}
\item[$(i)$]
$(R)$ is equivalent to the condition that $ w \leq \pi(w)$, for any $w \in W$,
\item[$(ii)$]
$(L)$ is equivalent to the condition that $\pi(w) \leq w$, for any $w \in W$,
\item[$(iii)$]
$(Fa)$ is equivalent to the condition that $\pi$ is an order embedding, i.e., if $\pi (u) \leq \pi(v)$ then $u \leq v$, for any $u, v \in W$,
\item[$(iv)$]
$(Fu)$ is equivalent to the surjectivity of $\pi$.
\end{description}
\end{lemma}
\begin{proof}
First, recall that by normality, $(x, y) \in R$ iff $x \leq \pi(y)$. We use this equivalence to rewrite all the aforementioned conditions in terms of $\pi$. For $(i)$, by normality, it is clear that $(w, w) \in R$ iff $w \leq \pi(w)$. Hence, there is nothing to prove. For $(ii)$, by normality, the condition $(L)$, i.e., $R \subseteq \; \leq$ simply states that
for all $u, v \in W$ if $v \leq \pi(u)$ then $v \leq u$. This is trivially equivalent to $\pi(u) \leq u$, for all $u \in W$. For $(iii)$, assume $(Fa)$. To prove that $\pi$ is an order embedding, suppose that $\pi(u) \leq \pi(v)$. By $(Fa)$, there exists $y \in W$ such that $(y, u) \in R$ and for any $z \in W$ if $(y, z) \in R$ then $u \leq z$. By normality, $y \leq \pi(u)$. Put $z=v$. Since  $\pi(u) \leq \pi(v)$, we have $y \leq \pi(v)$. Again, by normality, $(y, v) \in R$. Hence, $u \leq v$, by $(Fa)$. For the converse, suppose that $\pi$ is an order embedding. To prove $(Fa)$, we have to show that for any $x \in W$, there exists $y \in W$ such that $(y, x) \in R$ and for any $z \in W$, if $(y, z) \in R$ then $x \leq z$. Set $y=\pi(x)$. Since $y \leq \pi(x)$, by normality, $(y, x) \in R$ and if $(y, z) \in R$, meaning $y \leq \pi(z)$, we have $\pi(x) \leq \pi(z)$ which implies $x \leq z$, by the assumption that $\pi$ is an order embedding. For $(iv)$, assume that $(Fu)$ holds. We show that $\pi$ is surjective. By $(Fu)$, for any $x \in W$, there exists $y \in W$ such that $(x, y) \in R$ and for any $z \in W$ such that $(z, y) \in R$, we have $z \leq x$. We claim that $\pi(y)=x$. Since $(x, y) \in R$, we have $x \leq \pi(y)$. To show $\pi(y) \leq x$, put $z=\pi(y)$. Then, by normality, $(z, y) \in R$ and hence by $(Fu)$, we have $\pi(y)=z \leq x$. Therefore, $\pi(y)=x$. Conversely, assume that $\pi$ is surjective. We show that $(Fu)$ holds. By surjectivity, for any $x \in W$, there exists $y \in W$ such that $\pi(y)=x$. We claim this $y$ works for the condition $(Fu)$. First, by normality, it is clear that $(x, y) \in R$. Now, if $z \in W$ such that $(z, y) \in R$ then, by normality $z \leq \pi(y)=x$. Hence, $z \leq x$.
\end{proof}
Using Theorem \ref{NormalForConditions}, one can interpret the normal Kripke frame $(W, \leq, R)$ with the normality witness $\pi: W \to W$ as the \emph{dynamic poset} $(W, \leq, \pi)$ consisting of the state space $(W, \leq)$ and the dynamism $\pi$. Then, the faithfulness and fullness conditions that we abstracted from dynamic topological systems in Example \ref{TopologicalSpace} play a similar role here. The former means that the dynamism $\pi$ is an order embedding while the latter is its surjectivity.\\

Similar to the previous lemma, when the Kripke frames are normal, we can rewrite the conditions of being a Kripke morphism purely in terms of the normality witnesses of the frames. 
\begin{lemma}\label{NormalityForMorph}
Let $\mathcal{K}=(W, \leq, R)$ and $\mathcal{K}'=(W', \leq', R')$ be two normal Kripke frames with the normality witnesses $\pi$ and $\pi'$, respectively. Then, for an order-preserving map $f: W \to W'$, the following are equivalent:
\begin{description}
\item[$(i)$]
$f$ is a Kripke morphism,
\item[$(ii)$]
$f \circ \pi=\pi' \circ f$ and $\pi'^{-1}(\uparrow \!\!f(k))=f[\pi^{-1}(\uparrow \!\! k)]$, for any $k \in W$.
\end{description}
\end{lemma}
\begin{proof}
To prove $(ii)$ from $(i)$, we first prove $f(\pi(k))=\pi'(f(k))$, for any $k \in W$. By definition, $f$ maps $R$ into $R'$, which by normality means that if $x \leq \pi(y)$ then $f(x) \leq' \pi'(f(y))$. Hence, putting $x=\pi(k)$ and $y=k$, we reach $f(\pi (k)) \leq' \pi'(f(k))$. To prove $\pi'(f(k)) \leq' f(\pi (k))$, by the third condition in the definition of a Kripke morphism, we know that for any $l' \in W'$ such that $(l', f(k)) \in R'$, there exists $l \in W$ such that $(l, k) \in R$ and $f(l) \geq' l'$. Put $l'=\pi'(f(k))$. It is clear that $l' \leq' \pi'(f(k))$ and hence $(l', f(k)) \in R'$. Therefore, there exists $l \in W$ such that $l \leq \pi(k)$ and $l' \leq' f(l)$. Since, $f$ is order-preserving, we have  $f(l) \leq' f(\pi(k))$. Therefore, $\pi'(f(k))=l' \leq' f(\pi (k))$ which implies $\pi'(f(k))= f(\pi (k))$.\\
For the second part of $(ii)$, 
we first prove the equivalence between $(f(k), l') \in R'$ and $l' \in f[\pi^{-1}(\uparrow \!\! k)]$, for any $k \in W$ and any $l'\in W'$. The forward direction simply is the second condition in the definition of a Kripke morphism. For the other direction, if there exists $l \in W$ such that $f(l)=l'$ and $k \leq \pi(l)$, then $(k, l) \in R$, by the normality condition and as $f$ is a Kripke morphism, we have $(f(k), f(l)) \in R'$ which implies $(f(k), l') \in R'$. 
Now, to prove the second part of $(ii)$, it is enough to note that for any $k \in W$ and any $l'\in W'$, we have the following sequence of equivalences:
$l' \in \pi'^{-1}(\uparrow \!\! f(k))$ iff $f(k) \leq' \pi'(l')$ iff $(f(k), l') \in R'$ iff $l' \in f[\pi^{-1}(\uparrow \!\! k)]$. Therefore, $\pi'^{-1}(\uparrow \!\!f(k))=f[\pi^{-1}(\uparrow \!\! k)]$, for any $k \in W$.\\
To prove $(i)$ from $(ii)$, we have to check the three conditions in the definition of a Kripke morphism. For the first condition, if $(k, l) \in R$ then $k \leq \pi(l)$ which implies $f(k) \leq' f(\pi(l))=\pi'(f(l))$. Hence, $(f(k), f(l)) \in R'$. For the second condition, if $(f(k), l') \in R'$, then $f(k) \leq' \pi'(l')$. Hence, $l' \in \pi'^{-1}(\uparrow \!\! f(k))$ which implies $l' \in f[\pi^{-1}(\uparrow \!\! k)]$. Therefore, there exists $l \in W$ such that $l'=f(l)$ and $k \leq \pi(l)$ which implies $(k, l) \in R$. For the third condition, if  $(l', f(k)) \in R'$, we have $l' \leq' \pi'(f(k))=f(\pi(k))$. Put $l=\pi(k)$. Then, $f(l) \geq' l'$ and $(l, k) \in R$.
\end{proof}

As we already observed in Example \ref{KripkeFrames}, any Kripke frame gives rise to a $\nabla$-algebra in a canonical way. We now show that this canonical assignment is a functor and it preserves the conditions  $\{N, R, L, Fa, Fu\}$. For that purpose, define the assignment $\mathfrak{U}$ on the object $\mathcal{K}=(W, \leq, R)$ of the category $\mathbf{K}$ by $ \mathfrak{U}(\mathcal{K})=(U(W, \leq), \nabla_{\mathcal{K}}, \to_{\mathcal{K}})$, where $\nabla_{\mathcal{K}}(U)=\{x \in X \mid \exists y \in U, (y, x) \in R \}$ and $U \to_{\mathcal{K}} V=\{x \in X \mid \forall y \in U, [(x, y) \in R \Rightarrow y \in V] \}$. Moreover, define $\mathfrak{U}$ on the morphism $f: \mathcal{K} \to \mathcal{K}'$ of $\mathbf{K}$ by $\mathfrak{U}(f)=f^{-1}: \mathfrak{U}(\mathcal{K}') \to \mathfrak{U}(\mathcal{K})$. 
\begin{theorem}\label{UFunctor}
The assignment $\mathfrak{U}: \mathbf{K}^{op} \to \mathbf{Alg}_{\nabla}(H)$ is a functor and for any $C \subseteq \{N, R, L, Fa, Fu\}$, if $\mathcal{K} \in \mathbf{K}(C)$, then $\mathfrak{U}(\mathcal{K})$ lands in $\mathbf{Alg}_{\nabla}(C, H)$. Moreover, the restriction of the functor $\mathfrak{U}$ to $[\mathbf{K}^H(C)]^{op}$ lands in $\mathbf{Alg}^H_{\nabla}(C)$.
\end{theorem}
\begin{proof}
First, we study the object part of the functor. As explained in Example \ref{KripkeFrames}, the structure $\mathfrak{U}(W, \leq, R)$ is a Heyting $\nabla$-algebra. For the conditions $\{N, R, L\}$, we refer the reader to \cite{ImSpace}. For $(Fa)$, by Theorem \ref{Surjectivity}, it is enough to prove $\nabla \Box U=U$, for any upset $U \subseteq W$. From the adjunction $\nabla \dashv \Box$, it is clear that $\nabla \Box U \subseteq U$. For the converse, assume $x \in U$. Then, since $(W, \leq, R)$ satisfies $(Fa)$, there exists $y \in W$ such that $(y, x) \in R$ and for any $z \in W$ such that $(y, z) \in R$, we have $x \leq z$. We claim $y \in \Box U$, because for any $z \in W$ such that $(y, z) \in R$, we have $x \leq z$ and as $U$ is upward-closed and $x \in U$, we reach $z \in U$. Now, as $(y, x) \in R$ and $y \in \Box U$, we can conclude $x \in \nabla \Box U$.\\
For $(Fu)$, using Theorem \ref{Injectivity}, it is enough to prove $ \Box \nabla U=U$. Again, from the adjunction $\nabla \dashv \Box$, it is clear that $ U \subseteq \Box \nabla U$. For the converse, assume $x \in \Box \nabla U$. Since $(W, \leq, R)$ satisfies $(Fu)$,  there exists $y \in W$ such that $(x, y) \in R$ and for any $z \in W$ such that $(z, y) \in R$, we have $z \leq x$. Since $(x, y) \in R$, we have $y \in \nabla U$. Therefore, there exists $w \in W$ such that $(w, y) \in R$ and $w \in U$. Therefore, by $(Fu)$, we have $w \leq x$. Since $U$ is upward-closed, we have $x \in U$. \\
For the morphisms, we have to prove that if $f: (W, \leq, R) \to (W', \leq', R')$ is a Kripke morphism, then $\mathfrak{U}(f)=f^{-1}$ preserves $\nabla$ and the implication and if $f$ is also Heyting, $\mathfrak{U}(f)$ also preserves the Heyting implication. For $\nabla$, consider the following two series of equivalences:
\[
x \in f^{-1}(\nabla U) \;\;\; \text{iff} \;\;\; f(x) \in \nabla U \;\;\; \text{iff} \;\;\; \exists y' \in U, [(y', f(x)) \in R'],
\] 
and
\[
x \in \nabla f^{-1}(U) \;\;\; \text{iff} \;\;\; \exists y \in f^{-1}(U), [(y, x) \in R].
\] 
Therefore, it is enough to show the equivalence between $\exists y' \in U, [(y', f(x)) \in R']$ and $\exists y \in f^{-1}(U), [(y, x) \in R ]$. The latter proves the former easily, because $(y, x) \in R$ implies $(f(y), f(x)) \in R'$ and it is sufficient to put $y'=f(y)$. For the converse, if there exists $y' \in U$ such that $(y', f(x)) \in R'$, then by part $(iii)$ of the definition of Kripke morphisms, there exists $y \in W$ such that $f(y) \geq' y'$ and $(y, x) \in R$. The only thing to prove is $f(y) \in U$ which is a consequence of $y' \in U$, $f(y) \geq' y'$, and the upward-closedness of $U$.\\
The proof for the Heyting implication is well-known and the case for implication is similar to that of Heyting implication. Finally, it is clear
by its definition that $\mathfrak{U}$ preserves the identity and the composition.
\end{proof}

There is another functor transforming a distributive $\nabla$-algebra to a Kripke frame. To introduce this functor, we use the usual prime filter construction, extensively explained in \cite{ImSpace}. Here, we recall the construction, as we need the detailed explanation to establish the construction for faithful and full distributive $\nabla$-algebras and also to address the morphisms that are missing in \cite{ImSpace}. More importantly, the construction also plays the main role in the duality theory planned to be covered in the sequel of the present paper and hence deserves its comprehensive presentation.\\

\noindent \textbf{Prime Filter Construction.} Let $\mathcal{A}=(\mathsf{A}, \nabla, \to)$ be a distributive $\nabla$-algebra. Define $\mathfrak{P}(\mathcal{A})=(\mathcal{F}_p(\mathsf{A}), \subseteq, R_{\mathcal{A}})$, where $\mathcal{F}_p(\mathsf{A})$ is the set of all prime filters of $\mathsf{A}$ and the relation $R_{\mathcal{A}}$ defined by 
\begin{center}
$(P, Q) \in R_{\mathcal{A}}$ iff [($a \to b \in P$ and $a \in Q$) implies $b \in Q$], for any $a, b \in A$.    
\end{center}
Moreover, for any $\nabla$-algebra morphism $f: \mathcal{A} \to \mathcal{B}$ define $\mathfrak{P}(f)=f^{-1}: \mathfrak{P}(\mathcal{B}) \to \mathfrak{P}(\mathcal{A})$ and set $i_{\mathcal{A}}: \mathsf{A} \to U(\mathcal{F}_p(\mathsf{A}), \subseteq)$ as $i_{\mathcal{A}}(a)=\{P \in \mathcal{F}_p(\mathsf{A}) \mid a \in P\}$.

\begin{lemma} \label{RelationCharacterization}
$(P, Q) \in R_{\mathcal{A}}$ iff $\nabla [P]=\{\nabla x \mid x \in P\} \subseteq Q$, for any two prime filters $P$ and $Q$ of $\mathsf{A}$.
\end{lemma}
\begin{proof}
If $(P, Q) \in R_{\mathcal{A}}$ and $x \in P$, since $x \leq 1 \to \nabla x$ and $P$ is a filter, $1 \to \nabla x \in P$. Therefore, as $1 \in Q$ and $(P, Q) \in R_{\mathcal{A}}$, we reach $ \nabla x \in Q$. Conversely, if $\nabla[P] \subseteq Q$, $a \to b \in P$ and $a \in Q$, then $\nabla (a \to b) \in \nabla [P] \subseteq Q$ and since $a \wedge \nabla (a \to b) \leq b$, we have $b \in Q$.
\end{proof}

\begin{theorem}\label{KripkeEmbedding}
The assignment $\mathfrak{P}: \mathbf{Alg}_{\nabla}(D) \to \mathbf{K}^{op}$ is a functor and $i_{\mathcal{A}}: \mathcal{A} \to \mathfrak{U} \mathfrak{P}(\mathcal{A})$ is a $\nabla$-algebra embedding, natural in $\mathcal{A}$. Moreover, for any $C \subseteq \{N, R, L, Fa, Fu\}$, if $\mathcal{A} \in \mathbf{Alg}_{\nabla}(C, D)$, then $\mathfrak{P}(\mathcal{A})$ lands in $[\mathbf{K}(C)]^{op}$. It also maps $\mathbf{Alg}^H_{\nabla}(C)$ to $[\mathbf{K}^H(C)]^{op}$ where $i_{\mathcal{A}}$ becomes a Heyting $\nabla$-algebra morphism.
\end{theorem}
\begin{proof}
For the sake of readability, we split the proof into some parts. First, we focus on the objects and then we will cover the morphisms.

I. First, note that by Lemma \ref{RelationCharacterization}, it is easy to prove that $R_{\mathcal{A}}$ is compatible with the relation $\subseteq$ which implies that the structure $\mathfrak{P}(\mathcal{A})$ is actually a Kripke frame. To show that $i_{\mathcal{A}}$ is a $\nabla$-algebra embedding, note that the following three facts are well-known: first, $i_{\mathcal{A}}$ is a bounded lattice embedding, second, it is natural even on all distributive bounded lattices and the third, if $\mathcal{A}$ is a Heyting algebra, $i_{\mathcal{A}}$ also preserves the Heyting implication \cite{Esakia,Fitting}. Therefore, the only thing to check is the preservation of $\nabla$ and $\to$ which can be found in full detail in the proof of Theorem 17 in \cite{ImSpace}.

II. We show that $\mathfrak{P}$ preserves the conditions in the set $\{N, R, L, Fa, Fu\}$. For the conditions $\{N, R, L\}$, we refer the reader to the proof of Theorem 17 in \cite{ImSpace}. For the other two conditions, i.e., $(Fa)$ and $(Fu)$, we have the following:\\
For $(Fa)$, assume that $\mathcal{A}$ satisfies $(Fa)$. We first show that for any $P \in \mathcal{F}_p(\mathsf{A})$, there exists  $Q \in \mathcal{F}_p(\mathsf{A})$ such that $\Box [P] \subseteq Q \subseteq \nabla^{-1} (P)$, where $\Box [P]=\{\Box a \mid a \in P\}$. Define $F$ as the filter generated by $\Box [P]$ and $I$ as the ideal generated by $(\nabla^{-1} P)^c$. We have $F \cap I=\varnothing$. Because, if $x \in F \cap I$, there are $p_1, \cdots, p_n \in P$ and $y_1, \cdots, y_m \in (\nabla^{-1} P)^c$ such that $\bigwedge_{i=1}^n \Box p_i \leq x \leq \bigvee_{j=1}^m y_j$. Define $p=\bigwedge_{i=1}^n p_i$. Since $P$ is a filter, we have $p \in P$. Since $\Box$ commutes with the meets, we have $\Box p \leq \bigvee_{j=1}^m y_j$. Hence, $p=\nabla \Box p \leq \bigvee_{j=1}^m \nabla y_j$ and as $p \in P$ and $P$ is a filter, we have $\bigvee_{j=1}^m \nabla y_j \in P$. Since $P$ is a prime filter, for at least one $1 \leq j \leq m$, we have $\nabla y_j \in P$ which is impossible as $y_j \in (\nabla^{-1} P)^c$. Hence, $F \cap I=\varnothing$. Therefore, by the prime filter theorem, Theorem \ref{PrimeFilterTheorem}, there exists a prime filter $Q$ such that $F \subseteq Q$ and $Q \cap I=\varnothing$. By the first, we can prove $\Box [P] \subseteq Q$. By the second, we have $Q \subseteq \nabla^{-1} (P)$.\\
Now, we are ready to prove that $\mathfrak{P}({\mathcal{A}})$ satisfies the condition $(Fa)$. We have to show that for any prime filter $P$, there exists a prime filter $Q$ such that $(Q, P) \in R_{\mathcal{A}}$ and for any prime filter $M$, if $(Q, M) \in R_{\mathcal{A}}$, then $P \subseteq M$. To prove that, pick an arbitrary prime filter $P$ and set $Q$ as the prime filter constructed above. Since $Q \subseteq \nabla^{-1}(P)$, we have $\nabla[Q] \subseteq P$ which implies $(Q, P) \in R_{\mathcal{A}}$, by Lemma \ref{RelationCharacterization}. Let $M$ be a prime filter such that $(Q, M) \in R_{\mathcal{A}}$. Hence, $\nabla[Q] \subseteq M$, by Lemma \ref{RelationCharacterization}. To show that $P \subseteq M$, let $p \in P$. Then $p=\nabla \Box p$. Hence, $p \in \nabla [\Box [P]] \subseteq \nabla [Q] \subseteq M$. Therefore, $P \subseteq M$.\\
For $(Fu)$, assume that $\mathcal{A}$ satisfies $(Fu)$. We first show that for any $P \in \mathcal{F}_p(\mathsf{A})$, there exists  $Q \in \mathcal{F}_p(\mathsf{A})$ such that $\nabla [P] \subseteq Q$ and $\nabla^{-1}(Q) \subseteq P$. Define $F$ as the filter generated by $\nabla [P]$ and $I$ as the ideal generated by $\nabla [P^c]$. We have $F \cap I=\varnothing$. Because, if $x \in F \cap I$, there are $p_1, \cdots, p_n \in P$ and $y_1, \cdots, y_m \in P^c$ such that $\bigwedge_{i=1}^n \nabla p_i \leq x \leq \bigvee_{j=1}^m \nabla y_j$. Define $p=\bigwedge_{i=1}^n p_i$. Since $P$ is a filter, we have $p \in P$. Since $\nabla$ commutes with the joins and $\nabla p \leq \bigwedge_{i=1}^n \nabla p_i$, we have $\nabla p \leq \nabla (\bigvee_{j=1}^m y_j)$. Since $\mathcal{A}$ satisfies $(Fu)$, $\nabla$ is an order embedding by Theorem \ref{Injectivity}. Hence, we have $p \leq \bigvee_{j=1}^m y_j$. Since $P$ is a prime filter, for at least one $1 \leq j \leq m$, we have $ y_j \in P$ which is impossible as $y_j \in P^c$. Hence, $F \cap I=\varnothing$. Therefore, by the prime filter theorem, Theorem \ref{PrimeFilterTheorem}, there exists a prime filter $Q$ such that $F \subseteq Q$ and $Q \cap I=\varnothing$. By the first, we can prove $\nabla [P] \subseteq Q$. By the second, we have $\nabla^{-1} (Q) \subseteq P$, because if $x \in \nabla^{-1} (Q)$ and $x \notin P$, then $\nabla x \in Q \cap \nabla [P^c]$ which is impossible. \\ 
Now, we can prove that $\mathfrak{P}({\mathcal{A}})$ satisfies the condition $(Fu)$. We have to show that for any prime filter $P$, there exists a prime filter $Q$ such that $(P, Q) \in R_{\mathcal{A}}$ and for any prime filter $M$, if $(M, Q) \in R_{\mathcal{A}}$, then $M \subseteq P$. Let $P$ be an arbitrary prime filter. Pick $Q$ as the prime filter constructed above. Since $\nabla [P] \subseteq Q$, we have $(P, Q) \in R_{\mathcal{A}}$, by Lemma \ref{RelationCharacterization}. Let $M$ be a prime filter such that $(M, Q) \in R_{\mathcal{A}}$. We have to show that $M \subseteq P$. Since $(M, Q) \in R_{\mathcal{A}}$, by Lemma \ref{RelationCharacterization}, we have $\nabla [M] \subseteq Q$ and hence $M \subseteq \nabla^{-1}(Q) \subseteq P$.

III. To address the morphisms, i.e., to show that if $f: \mathcal{A} \to \mathcal{B}$ is a $\nabla$-algebra morphism, $\mathfrak{P}(f)=f^{-1}:(\mathcal{F}_p(\mathsf{B}), \subseteq, R_{\mathcal{B}}) \to (\mathcal{F}_p(\mathsf{A}), \subseteq, R_{\mathcal{A}})$ is a Kripke morphism, we have to check the three conditions in the definition of a Kripke morphism.
First, we show that $(P, Q) \in R_{\mathcal{B}}$ implies $(f^{-1}(P), f^{-1}(Q)) \in R_{\mathcal{A}}$, for any $P, Q \in \mathcal{F}_p(\mathsf{B})$. First, note that by the fact that $f$ preserves $\nabla$, we have $\nabla_{\mathcal{A}} [f^{-1}(P)] \subseteq f^{-1} (\nabla_{\mathcal{B}} [P])$. The reason is that if $x \in f^{-1}(P)$, then as $f(\nabla_{\mathcal{A}} x)=\nabla_{\mathcal{B}} f(x) \in \nabla_{\mathcal{B}} [P]$, we have $\nabla_{\mathcal{A}} x \in f^{-1} (\nabla_{\mathcal{B}} [P])$. Now, assume $(P, Q) \in R_{\mathcal{B}}$. By Lemma \ref{RelationCharacterization}, we have $\nabla_{\mathcal{B}} [P] \subseteq Q$ which implies 
$
\nabla_{\mathcal{A}} [f^{-1}(P)] \subseteq f^{-1}(\nabla_{\mathcal{B}} [P]) \subseteq f^{-1} (Q)$,
and hence $(f^{-1}(P), f^{-1}(Q)) \in R_{\mathcal{A}}$, by Lemma \ref{RelationCharacterization}.\\
For the second condition, assume $(f^{-1}(P'), Q) \in R_{\mathcal{A}}$, for some $P' \in \mathcal{F}_p(\mathsf{B})$ and $Q \in \mathcal{F}_p(\mathsf{A})$. We have to provide $Q' \in \mathcal{F}_p(\mathsf{B})$ such that $(P', Q') \in R_{\mathcal{B}}$ and $f^{-1}(Q')=Q$. Define $I$ as the ideal generated by $f[Q^c]$ and $F$ as the filter generated by $\nabla_{\mathcal{B}}[P'] \cup f[Q]$. We claim that $F \cap I=\varnothing$. Assume $x \in F \cap I$. Then, there exist $y_1, \cdots, y_m \in Q^c$, $z_1, \cdots, z_n \in Q$ and $w_1, \cdots, w_k \in P'$ such that $\bigwedge_{j=1}^n f(z_j) \wedge \bigwedge_{r=1}^k \nabla_{\mathcal{B}} w_r \leq x \leq \bigvee_{i=1}^m f(y_i) $. Define $z=\bigwedge_{j=1}^n z_j$, $w=\bigwedge_{r=1}^k w_r$ and note that $z \in Q$ and $w \in P'$, since both $Q$ and $P'$ are filters. Then, by the monotonicity of $\nabla_{\mathcal{B}}$ and the fact that $f$ is a $\nabla$-algebra morphism, we have $\nabla_{\mathcal{B}} w \wedge f(z) \leq f(\bigvee_{i=1}^m y_i)$. Therefore, $w \leq_{\mathcal{B}} f(z) \to_{\mathcal{B}} f(\bigvee_{i=1}^m y_i)=f(z \to_{\mathcal{A}} \bigvee_{i=1}^m y_i)$. Since $w \in P'$, we have $f(z \to_{\mathcal{A}} \bigvee_{i=1}^m y_i) \in P'$ which implies $z \to_{\mathcal{A}} \bigvee_{i=1}^m y_i \in f^{-1}(P')$ and hence $\nabla_{\mathcal{A}}(z \to_{\mathcal{A}} \bigvee_{i=1}^m y_i) \in \nabla_{\mathcal{A}} [f^{-1}(P')]$. Since $(f^{-1}(P'), Q) \in R_{\mathcal{A}}$, by Lemma \ref{RelationCharacterization}, we have $\nabla_{\mathcal{A}} [f^{-1}(P')] \subseteq Q$. Hence, $\nabla_{\mathcal{A}}(z \to_{\mathcal{A}} \bigvee_{i=1}^m y_i) \in Q$. Since $z \in Q$ and $z \wedge \nabla_{\mathcal{A}}(z \to_{\mathcal{A}} \bigvee_{i=1}^m y_i) \leq \bigvee_{i=1}^m y_i$, we have $\bigvee_{i=1}^m y_i \in Q$. Since $Q$ is prime, for some $1 \leq i \leq m$, we must have $y_i \in Q$ which is a contradiction with the choice of $y_i$. Therefore, $F \cap I = \varnothing$. Now, by the prime filter theorem, Theorem \ref{PrimeFilterTheorem}, there exists a prime filter $Q' \in \mathcal{F}_p(\mathsf{B})$ such that $F \subseteq Q'$ and $Q' \cap I=\varnothing$. The first implies $\nabla_{\mathcal{B}}[P'] \cup f[Q] \subseteq Q'$ which also implies $(P', Q') \in R_{\mathcal{B}}$ and $Q \subseteq f^{-1}(Q')$. From the second, we have $f^{-1}(Q') \subseteq Q$, because if $x \in f^{-1}(Q')$ and $x \notin Q$, then $x \in Q^c$ and hence $f(x) \in I$ which implies the contradictory result $f(x) \in Q' \cap I$. Hence, $f^{-1}(Q')=Q$.\\
For the third condition, let us assume $(Q, f^{-1}(P')) \in R_{\mathcal{A}}$, for some $P' \in \mathcal{F}_p(\mathsf{B})$ and $Q \in \mathcal{F}_p(\mathsf{A})$. We have to find a prime filter $Q' \in \mathcal{F}_p(\mathsf{B})$ such that $(Q', P') \in R_{\mathcal{B}}$ and $Q \subseteq f^{-1}(Q')$. Define $I$ as the ideal generated by $(\nabla ^{-1}_{\mathcal{B}} [P'])^c$ and $F$ as the filter generated by $f[Q]$. We have $F \cap I=\varnothing$, because if $x \in I \cap F$, then, there exist $y_1, \cdots, y_m \notin \nabla^{-1}_{\mathcal{B}} [P']$ and $z_1, \cdots, z_n \in Q$ such that $\bigwedge_{i=1}^n f(z_i) \leq x \leq \bigvee_{j=1}^m y_j$. Define $z=\bigwedge_{i=1}^n z_i$ and note that $z \in Q$, since $Q$ is a filter. Since $f$ is a $\nabla$-algebra morphism, $f(z) \leq \bigvee_{j=1}^m y_j$. Therefore, $f(\nabla_{\mathcal{A}} z)=\nabla_{\mathcal{B}} f(z) \leq \bigvee_{j=1}^m \nabla_{\mathcal{B}} y_j$. Since $(Q, f^{-1}(P')) \in R_{\mathcal{A}}$, by Lemma \ref{RelationCharacterization}, we have $\nabla_{\mathcal{A}}[Q] \subseteq f^{-1}(P')$. Since $z \in Q$, we have $f(\nabla_{\mathcal{A}} z) \in P'$ and hence $\bigvee_{j=1}^m \nabla_{\mathcal{B}} y_j \in P'$. Since $P'$ is prime, for some $1 \leq j \leq m$, we have $\nabla_{\mathcal{B}} y_j \in P'$ which contradicts  $y_j \notin \nabla^{-1}_{\mathcal{B}} [P']$. Hence, $F \cap I=\varnothing$. Therefore, by the prime filter theorem, Theorem \ref{PrimeFilterTheorem}, there exists a prime filter $Q' \in \mathcal{F}_p(\mathcal{B})$ such that $F \subseteq Q'$ and $Q'\cap I=\varnothing$. The first implies $f[Q] \subseteq Q'$ and hence $Q \subseteq f^{-1}(Q')$ and the second implies $Q' \subseteq \nabla^{-1}_{\mathcal{B}}(P')$ which means $\nabla_{\mathcal{B}}[Q'] \subseteq P'$ and by Lemma \ref{RelationCharacterization}, $(Q', P') \in R_{\mathcal{B}}$. \\
For the Heyting $\nabla$-algebra morphisms, it is a well-known fact that if $f$ preserves the Heyting implication, then $f^{-1}$ satisfies the fourth condition in the definition of Heyting Kripke morphisms, see Theorem 3.3.4 in \cite{Esakia}. Finally, note that $\mathfrak{P}$ clearly preserves the identity and the composition.
\end{proof}

\subsection{Amalgamation Property}\label{SectionAmalgamation}
In this subsection, we use the provided Kripke representation of distributive $\nabla$-algebras to prove the amalgamation property for the varieties $\mathcal{V}(C, D, N)$ and $\mathcal{V}_H(C, N)$, for any $C \subseteq \{R, L, Fa\}$.
\begin{definition}
A class $\mathfrak{C}$ of $\nabla$-algebras has the {\em amalgamation property} if for any $\mathcal{A}_{0}, \mathcal{A}_1$ and $\mathcal{A}_2$ in $\mathfrak{C}$ and any $\nabla$-algebra embeddings $f_{1}:\mathcal{A}_0 \longrightarrow \mathcal{A}_1$ and $f_{2}: \mathcal{A}_0 \longrightarrow \mathcal{A}_2$, there exist a $\nabla$-algebra $\mathcal{B}$ in $\mathfrak{C}$ and $\nabla$-algebra embeddings $g_{1}: \mathcal{A}_1 \longrightarrow \mathcal{B}$ and $g_{2}: \mathcal{A}_2 \longrightarrow \mathcal{B}$ such that $g_{1} \circ f_{1}=g_{2}\circ f_{2}$:
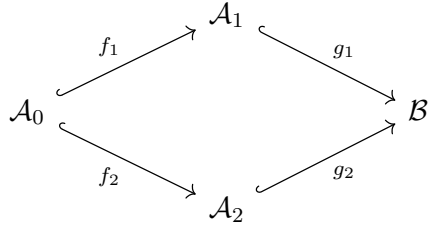
\[\small \begin{tikzcd}
	&& {\mathcal{A}_1} \\
	{\mathcal{A}_0} &&&& {\mathcal{B}} \\
	&& {\mathcal{A}_2}
	\arrow["{f_1}", hook, from=2-1, to=1-3]
	\arrow["{f_2}"', hook, from=2-1, to=3-3]
	\arrow["{g_1}", hook, from=1-3, to=2-5]
	\arrow["{g_2}"', hook, from=3-3, to=2-5]
\end{tikzcd}\]
The amalgamation property for a class of explicitly Heyting $\nabla$-algebras is defined similarly, replacing $\nabla$-algebra morphisms by Heyting $\nabla$-algebra morphisms, everywhere.
\end{definition}
Our strategy to prove the amalgamation property is to transform the property to the dual property for the Kripke frames where constructing the new frames is rather easier. For that purpose, we need the following well-known fact for which we also provide a proof for the sake of completeness:
\begin{lemma}\label{SurjToInj}
Let $\mathsf{A}$ and $\mathsf{B}$ be two bounded distributive lattices and $f: \mathsf{A} \to \mathsf{B}$ be an injective bounded lattice map. Then, the map $f^{-1}: \mathcal{F}_{p}(\mathsf{B)} \to \mathcal{F}_{p}(\mathsf{A)}$ is surjective. Therefore, the functor $\mathfrak{P}$ maps the injective $\nabla$-algebra morphisms to surjective Kripke morphisms. Conversely, the functor $\mathfrak{U}$ maps surjective Kripke morphisms to injective $\nabla$-algebra morphisms.
\end{lemma}
\begin{proof}
The last part is an easy consequence of the fact that if $f: W_1 \to W_2$ is surjective, then $\mathfrak{U}(f)=f^{-1}$ is one-to-one. For the first part, assume that $f: \mathsf{A} \to \mathsf{B}$ is a one-to-one bounded lattice morphism. Then, we want to prove that for any $P \in \mathcal{F}_p(\mathsf{A})$, there exists $Q \in \mathcal{F}_p(\mathsf{B})$ such that $f^{-1}(Q)=P$. Define $F$ as the filter generated by $f[P]$ and $I$ as the ideal generated by $f[P^c]$. We have $F \cap I=\varnothing$, because if $x \in F \cap I$, then there exist $p_1, \cdots, p_m \in P$ and $y_1, \cdots, y_n \in P^c$ such that $\bigwedge_{i=1}^m f(p_i) \leq x \leq \bigvee_{j=1}^n f(y_j)$. Since $f$ is a bounded lattice morphism, we have $f(\bigwedge_{i=1}^m p_i) \leq f(\bigvee_{j=1}^n y_j)$. Since $f$ is one-to-one, it is an order embedding and hence, we have $\bigwedge_{i=1}^m p_i \leq \bigvee_{j=1}^n y_j$. Since $P$ is a prime filter, for at least one $1 \leq j \leq n$, we must have $y_j \in P$ which contradicts with the choice of $y_j$. Hence, $F \cap I=\varnothing$. Now, by the prime filter theorem, Theorem \ref{PrimeFilterTheorem}, there exists a prime filter $Q \in \mathcal{F}_p(\mathsf{B})$ such that $F \subseteq Q$ and $Q \cap I=\varnothing$. From the first, we have $f[P] \subseteq Q$ which implies $P \subseteq f^{-1}(Q)$ and from the second $f^{-1}(Q) \subseteq P$. Hence, $f^{-1}(Q)=P$.
\end{proof}
The next lemma is the dual property we promised to prove.
\begin{lemma}\label{AmalgamtionForKripke}
Let $C \subseteq \{R, L, Fa\}$. Then, for any normal Kripke frames $\mathcal{K}_0=(W_0, \leq_0, R_0)$, $\mathcal{K}_1=(W_1, \leq_1, R_1)$ and $\mathcal{K}_2=(W_2, \leq_2, R_2)$ in $\mathbf{K}(C, N)$ and any surjective Kripke morphisms $f: \mathcal{K}_1 \to \mathcal{K}_0$ and $g: \mathcal{K}_2 \to \mathcal{K}_0$, there exist a Kripke frame $\mathcal{K} \in \mathbf{K}(C, N)$ and surjective Kripke morphisms $p: \mathcal{K} \to \mathcal{K}_1$ and $q: \mathcal{K} \to \mathcal{K}_2$ such that $f \circ p=g \circ q$:
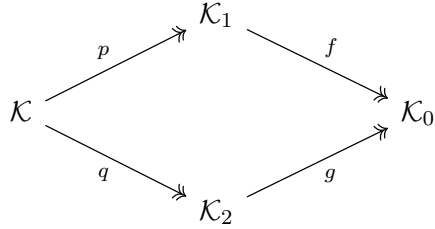
\[\small \begin{tikzcd}
	&& {\mathcal{K}_1} \\
	{\mathcal{K}} &&&& {\mathcal{K}_0} \\
	&& {\mathcal{K}_2}
	\arrow["p", two heads, from=2-1, to=1-3]
	\arrow["q"', two heads, from=2-1, to=3-3]
	\arrow["f", two heads, from=1-3, to=2-5]
	\arrow["g"', two heads, from=3-3, to=2-5]
\end{tikzcd}\]
Moreover, if $f$ and $g$ are Heyting, so are $p$ and $q$.
\end{lemma}
\begin{proof}
Define 
$
W= \{ (y, z) \in W_1 \times W_2 \mid f(y)=g(z) \}
$, $\leq \; =(\leq_1 \times \leq_2 )|_{W}$ and $R = (R_1 \times R_2)|_{W}$. It is easy to see that $\mathcal{K}=(W, \leq, R)$ is a Kripke frame, as $R$ is clearly compatible with $\leq$. To prove that $\mathcal{K}$ is normal, assume that $\pi_1: W_1 \to W_1$ and $\pi_2: W_2 \to W_2$ are the normality witnesses of $\mathcal{K}_1$ and $\mathcal{K}_2$, respectively. Define $\pi: W \to W$ by $\pi(y, z)=(\pi_1(y), \pi_2(z))$. It is well-defined, i.e., $\pi(y, z) \in W$, for any $(y, z) \in W$. Because, $(y, z) \in W$ implies $f(y)=g(z)$ and since $f$ and $g$ are Kripke morphisms, by Lemma \ref{NormalityForMorph}, we have $f(\pi_1(y))=\pi_0(f(y))=\pi_0(g(z))=g(\pi_2(z))$ 
which implies $(\pi_1(y), \pi_2 (z)) \in W$. The function $\pi$ also respects the order $\leq$ and we have $((y, z), (y', z')) \in R$ iff $(y, z) \leq \pi(y', z')$. Hence, $\mathcal{K}$ is normal. To show that $\mathcal{K} \in \mathbf{K}(C)$, using Lemma \ref{NormalForConditions}, the conditions $\{L, R, Fa\}$ for $\mathcal{K}$ are equivalent to $\pi(y, z) \leq (y, z)$, $(y, z) \leq \pi(y, z)$ and the condition that $\pi$ is an order embedding, respectively. All the three are inherited from $\mathcal{K}_1$ and $\mathcal{K}_2$ to their product and since they are universal conditions, to the frame $\mathcal{K}$.\\
Now, set $p: W \to W_1$ and $q: W \to W_2$ as the projection functions. By the definition of $W$, it is clear that $f \circ p=g \circ q$. We only need to prove that $p$ and $q$
are surjective Kripke morphisms. We prove the claim for $p$. The case for $q$ is similar. First, note that $p$ is surjective, because for any $y \in W_1$, there exists $z \in W_2$ such that $f(y)=g(z)$, simply because $g$ is surjective. Therefore, $(y, z) \in W$ and as $p(y, z)=y$, the surjectivity of $p$ is established. Secondly, $p$ is a Kripke morphism. To prove, by Lemma \ref{NormalityForMorph}, it is enough to show that $p \circ \pi=\pi_1 \circ p$ and $\pi_1^{-1}(\uparrow \!\!p(x,y))=p[\pi^{-1}(\uparrow \!\! (x,y))]$, for any $(x, y) \in W$. The first part, i.e., $p \circ \pi=\pi_1 \circ p$ is clear from the definition. For the second, for any $(x, y) \in W$ and $z \in W_1$, note that we have the following equivalences:
\[
z \in \pi_1^{-1}(\uparrow \!\!p(x,y)) \quad \text{iff} \quad x \leq_1 \pi_1(z),
\]
\[
z \in p[\pi^{-1}(\uparrow \!\! (x,y))] \quad \text{iff} \quad \exists u \in W_2 [(z, u) \in W \, \& \, (x, y) \leq (\pi_1(z), \pi_2(u))]. 
\]
Therefore, it is enough to prove the equivalence between the two right hand sides. The upward direction is clear. For the downward direction, if  $x \leq_1 \pi_1(z)$, then $(x, z) \in R_1$. As $f$ is a Kripke morphim, $(f(x), f(z)) \in R_0$. As $(x, y) \in W$, we have $f(x)=g(y)$. Hence, $(g(y), f(z)) \in R_0$. As $g$ is a Kripke morphism, there must be $u \in W_2$ such that $(y, u) \in R_2$ and $g(u)=f(z)$. Therefore, $(z, u) \in W$ and $y \leq_2 \pi_2(u)$ which completes the proof.
The proof to show that $p$ and $q$ are Heyting, if $f$ and $g$ are Heyting is similar.
\end{proof}

\begin{theorem}\label{Amalgamation}(Amalgamation) Let $C \subseteq \{R, L, Fa\}$. Then, the varieties $\mathcal{V}(C, D, N)$ and $\mathcal{V}_H(C, N)$ have the amalgamation property.
\end{theorem}
\begin{proof}
We first prove the claim for $\mathcal{V}(C, D, N)$. Let $\mathcal{A}_{0}, \mathcal{A}_1$ and $\mathcal{A}_2$ be some $\nabla$-algebras in $\mathcal{V}(C, D, N)$ and $f_{1}:\mathcal{A}_0 \longrightarrow \mathcal{A}_1$ and $f_{2}: \mathcal{A}_0 \longrightarrow \mathcal{A}_2$ are some embeddings. Applying the functor $\mathfrak{P}$, by Theorem \ref{KripkeEmbedding}, we reach the normal Kripke models $\mathcal{K}_i=\mathfrak{P}(\mathcal{A}_i)$ in $\mathbf{K}(C, N)$, for $i=0, 1, 2$ and Kripke morphisms $f=\mathfrak{P}(f_1): \mathcal{K}_1 \to \mathcal{K}_0$ and $g=\mathfrak{P}(f_2): \mathcal{K}_2 \to \mathcal{K}_0$. By Lemma \ref{SurjToInj}, we know that the maps $f$ and $g$ are surjective. By Lemma \ref{AmalgamtionForKripke}, there exist a Kripke frame $\mathcal{K} \in \mathbf{K}(C, N)$ and surjective Kripke morphisms $p: \mathcal{K} \to \mathcal{K}_1$ and $q: \mathcal{K} \to \mathcal{K}_2$ such that $f \circ p=g \circ q$:
\[\small \begin{tikzcd}
	&& {\mathcal{K}_1=\mathfrak{P}(\mathcal{A}_1)} \\
	{\mathcal{K}} &&&& {\mathcal{K}_0=\mathfrak{P}(\mathcal{A}_0)} \\
	&& {\mathcal{K}_2=\mathfrak{P}(\mathcal{A}_2)}
	\arrow["p", two heads, from=2-1, to=1-3]
	\arrow["q"', two heads, from=2-1, to=3-3]
	\arrow["{f=\mathfrak{P}(f_1)}", two heads, from=1-3, to=2-5]
	\arrow["{g=\mathfrak{P}(f_2)}"', two heads, from=3-3, to=2-5]
\end{tikzcd}\]
Now, apply the functor $\mathfrak{U}$ to the previous diagram and use Thoerem \ref{UFunctor} to land inside $\mathcal{V}(C, D, N)$. By Lemma \ref{SurjToInj}, the maps $\mathfrak{U}(p)$ and $\mathfrak{U}(q)$ are embeddings. Therefore, as $i_{\mathcal{A}}: \mathcal{A} \to \mathfrak{U}\mathfrak{P}(\mathcal{A})$ is an embedding and also natural in $\mathcal{A}$, we have the following commutative diagram of $\nabla$-algebras in $\mathcal{V}(C, D, N)$:
\[\small \begin{tikzcd}
	{\mathcal{A}_0} &&& {\mathcal{A}_2} \\
	& {\mathfrak{UP}(\mathcal{A}_0)} &&& {\mathfrak{UP}(\mathcal{A}_2)} \\
	{\mathcal{A}_1} \\
	& {\mathfrak{UP}(\mathcal{A}_1)} &&& {\mathfrak{U}(\mathcal{K})}
	\arrow["{\mathfrak{UP}(f_2)}", hook, from=2-2, to=2-5]
	\arrow["{\mathfrak{UP}(f_1)}", hook, from=2-2, to=4-2]
	\arrow["{\mathfrak{U}(q)}", hook, from=2-5, to=4-5]
	\arrow["{\mathfrak{U}(p)}"', hook, from=4-2, to=4-5]
	\arrow["{i_{\mathcal{A}_0}}", hook, from=1-1, to=2-2]
	\arrow["{i_{\mathcal{A}_1}}", hook, from=3-1, to=4-2]
	\arrow["{i_{\mathcal{A}_2}}", hook, from=1-4, to=2-5]
	\arrow["{f_1}"', hook, from=1-1, to=3-1]
	\arrow["{f_2}", hook, from=1-1, to=1-4]
\end{tikzcd}\]
Now, it is enough to set $\mathcal{A}=\mathfrak{U}(\mathcal{K})$ and $g_1=\mathfrak{U}(p) \circ i_{\mathcal{A}_1}$ and $g_2=\mathfrak{U}(q) \circ i_{\mathcal{A}_2}$.\\
Finally, to address the case $\mathcal{V}_H(C, N)$, note that if $f_1$ and $f_2$ preserve the Heyting implication, then by Theorem \ref{KripkeEmbedding}, the maps $\mathfrak{P}(f_1)$ and $\mathfrak{P}(f_2)$ are Heyting. Then, by Lemma \ref{AmalgamtionForKripke}, the maps $p$ and $q$ are Heyting Kripke morphisms which implies that the maps $\mathfrak{U}(p)$ and $\mathfrak{U}(q)$ preserve the Heyting implication, by Theorem \ref{UFunctor}. Finally, as $i_{\mathcal{A}_1}$ and $i_{\mathcal{A}_2}$ also preserve the Heyting implication, by Theorem \ref{KripkeEmbedding},
the maps $g_1$ and $g_2$ preserve the Heyting implication, as well.
\end{proof}

\noindent \textbf{Acknowledgments:}
We wish to thank Nick Bezhanishvili for his helpful suggestions and the anonymous referees whose suggestions helped to improve our presentation. The first author also gratefully acknowledges the support of the FWF project P 33548. He is also supported by the Czech Academy of Sciences (RVO 67985840).

\bibliographystyle{plainurl}
\bibliography{newbib}

\begin{thebibliography}{10}

\bibitem{ImSpace}
Amirhossein Akbar~Tabatabai.
\newblock Implication via spacetime.
\newblock {\em Mathematics, Logic, and their Philosophies: Essays in Honour of Mohammad Ardeshir}, pages 161--216, 2021.

\bibitem{Akin}
Ethan Akin.
\newblock The general topology of dynamical systems.
\newblock {\em Graduate Studies in Mathematics}, 2010.

\bibitem{Ard2}
Mohammad Ardeshir.
\newblock {\em Aspects of Basic Logic}.
\newblock PhD thesis, Marquette University, 1995.

\bibitem{BasicPropLogic}
Mohammad Ardeshir and Wim Ruitenburg.
\newblock Basic propositional calculus. {I}.
\newblock {\em MLQ Math. Log. Q.}, 44(3):317--343, 1998.

\bibitem{Lattar}
Mohammad Ardeshir and Wim Ruitenburg.
\newblock Latarres, lattices with an arrow.
\newblock {\em Studia Logica}, 106(4):757--788, 2018.

\bibitem{artemov1997modal}
SN~Artemov, JM~Davoren, and A~Nerode.
\newblock Modal logics and topological semantics for hybrid systems.
\newblock 1997.

\bibitem{balbiani2018bisimulations}
Philippe Balbiani, Joseph Boudou, Mart\'{\i}n Di\'{e}guez, and David Fern\'{a}ndez-Duque.
\newblock Bisimulations for intuitionistic temporal logics.
\newblock {\em J. Appl. Logics}, 8(8):2265--2285, 2021.

\bibitem{Bezh}
Guram Bezhanishvili, Nick Bezhanishvili, David Gabelaia, and Alexander Kurz.
\newblock Bitopological duality for distributive lattices and {H}eyting algebras.
\newblock {\em Math. Structures Comput. Sci.}, 20(3):359--393, 2010.

\bibitem{Bor1}
Francis Borceux.
\newblock Handbook of categorical algebra 1.
\newblock {\em Encyclopedia of Mathematics and its Applications, Vol. 50, Cambridge University Press, Cambridge}, 1994.

\bibitem{boudou2017decidable}
Joseph Boudou, Mart{\'\i}n Di{\'e}guez, and David Fern{\'a}ndez-Duque.
\newblock A decidable intuitionistic temporal logic.
\newblock In {\em 26th EACSL annual conference on Computer Science Logic (CSL 2017)}, volume~82. Schloss Dagstuhl. Leibniz-Zentrum f{\"u}r Informatik, 2017.

\bibitem{Sank}
Stanley Burris and H.~P. Sankappanavar.
\newblock A course in universal algebra.
\newblock {\em Graduate Texts in Mathematics, Vol. 78, Springer-Verlag, New York-Berlin}, 1981.

\bibitem{Celani-Jansana}
Sergio Celani and Ramon Jansana.
\newblock Bounded distributive lattices with strict implication.
\newblock {\em MLQ Math. Log. Q.}, 51(3):219--246, 2005.

\bibitem{Chajda}
Ivan Chajda.
\newblock Algebraic axiomatization of tense intuitionistic logic.
\newblock {\em Cent. Eur. J. Math.}, 9(5):1185--1191, 2011.

\bibitem{Pal}
Willem Conradie, Alessandra Palmigiano, Sumit Sourabh, and Zhiguang Zhao.
\newblock Canonicity and relativized canonicity via pseudo-correspondence: an application of alba.
\newblock {\em arXiv preprint arXiv:1511.04271}, 2015.

\bibitem{NuteCross}
Charles~B Cross.
\newblock Nute donald. conditional logic. handbook of philosophical logic, volume ii, extensions of classical logic, edited by gabbay d. and guenthner f., synthese library, vol. 165, d. reidel publishing company, dordrecht, boston, and lancaster, 1984, pp. 387--439.
\newblock 1989.

\bibitem{Davies2017}
Rowan Davies.
\newblock A temporal logic approach to binding-time analysis.
\newblock {\em J. ACM}, 64(1):Art. 1, 45, 2017.

\bibitem{de2017constructive}
Valeria de~Paiva and Harley Eades~III.
\newblock Constructive temporal logic, categorically.
\newblock {\em IfCoLog Journal of Logics and their Applications}, page 1287, 2017.

\bibitem{ModEsakia}
Leo Esakia.
\newblock The modalized {H}eyting calculus: a conservative modal extension of the intuitionistic logic.
\newblock {\em J. Appl. Non-Classical Logics}, 16(3-4):349--366, 2006.

\bibitem{Esakia}
Leo Esakia.
\newblock Heyting algebras.
\newblock {\em Trends in Logic---Studia Logica Library, Vol. 50}, 2019.

\bibitem{Ewald}
W.~B. Ewald.
\newblock Intuitionistic tense and modal logic.
\newblock {\em J. Symbolic Logic}, 51(1):166--179, 1986.

\bibitem{Fer}
David Fern\'{a}ndez-Duque.
\newblock The intuitionistic temporal logic of dynamical systems.
\newblock {\em Log. Methods Comput. Sci.}, 14(3):Paper No. 3, 35, 2018.

\bibitem{Servi}
Gis\`ele Fischer-Servi.
\newblock On modal logic with an intuitionistic base.
\newblock {\em Studia Logica}, 36(3):141--149, 1977.

\bibitem{Fitting}
Melvin~Chris Fitting.
\newblock Intuitionistic logic, model theory and forcing.
\newblock {\em Studies in Logic and the Foundations of Mathematics, North-Holland Publishing Co., Amsterdam-London}, page 191, 1969.

\bibitem{Ghil}
Silvio Ghilardi and Giancarlo Meloni.
\newblock Constructive canonicity in non-classical logics.
\newblock {\em Ann. Pure Appl. Logic}, 86(1):1--32, 1997.

\bibitem{HardBezh}
John Harding and Guram Bezhanishvili.
\newblock Mac{N}eille completions of {H}eyting algebras.
\newblock {\em Houston J. Math.}, 30(4):937--952, 2004.

\bibitem{Iem1}
Rosalie Iemhoff.
\newblock Preservativity logic: an analogue of interpretability logic for constructive theories.
\newblock {\em MLQ Math. Log. Q.}, 49(3):230--249, 2003.

\bibitem{Iem2}
Rosalie Iemhoff, Dick de~Jongh, and Chunlai Zhou.
\newblock Properties of intuitionistic provability and preservativity logics.
\newblock {\em Log. J. IGPL}, 13(6):615--636, 2005.

\bibitem{StoneSpaces}
Peter~T. Johnstone.
\newblock Stone spaces.
\newblock {\em Cambridge Studies in Advanced Mathematics, Vol. 3, Cambridge University Press, Cambridge}, 1982.

\bibitem{Kamide}
Norihiro Kamide and Heinrich Wansing.
\newblock Combining linear-time temporal logic with constructiveness and paraconsistency.
\newblock {\em J. Appl. Log.}, 8(1):33--61, 2010.

\bibitem{Koj}
Kensuke Kojima and Atsushi Igarashi.
\newblock Constructive linear-time temporal logic: proof systems and {K}ripke semantics.
\newblock {\em Inform. and Comput.}, 209(12):1491--1503, 2011.

\bibitem{Und}
Boris Konev, Roman Kontchakov, Frank Wolter, and Michael Zakharyaschev.
\newblock Dynamic topological logics over spaces with continuous functions.
\newblock {\em Advances in modal logic}, 6:299--318, 2006.

\bibitem{kremer2004small}
Philip Kremer.
\newblock A small counterexample in intuitionistic dynamic topological logic, 2004.

\bibitem{DTL}
Philip Kremer and Grigori Mints.
\newblock Dynamic topological logic.
\newblock In {\em Handbook of Spatial Logics}, pages 565--606. Springer.

\bibitem{LitViss}
Tadeusz Litak and Albert Visser.
\newblock Lewis meets {B}rouwer: constructive strict implication.
\newblock {\em Indag. Math. (N.S.)}, 29(1):36--90, 2018.

\bibitem{Maier}
Patrick Maier.
\newblock Intuitionistic {LTL} and a new characterization of safety and liveness.
\newblock In {\em Computer Science Logic: 18th International Workshop, CSL 2004, 13th Annual Conference of the EACSL, Karpacz, Poland, September 20-24, 2004. Proceedings 18}, pages 295--309. Springer, 2004.

\bibitem{Okada}
Mitsuhiro Okada.
\newblock A weak intuitionistic propositional logic with purely constructive implication.
\newblock {\em Studia Logica}, 46(4):371--382, 1987.

\bibitem{Ru}
Wim Ruitenburg.
\newblock Basic logic and fregean set theory.
\newblock {\em Dirk van Dalen Festschrift, Questiones Infinitae}, 5:121--142, 1992.

\bibitem{Sim}
Alex~K Simpson.
\newblock The proof theory and semantics of intuitionistic modal logic.
\newblock 1994.

\bibitem{Vi2}
Albert Visser.
\newblock {\em Aspects of Diagonalization and Provability}.
\newblock PhD thesis, University of Utrecht, 1981.

\bibitem{visser1981propositional}
Albert Visser.
\newblock A propositional logic with explicit fixed points.
\newblock {\em Studia Logica}, pages 155--175, 1981.

\end{thebibliography}

\end{document}